\documentclass[a4paper,10pt]{article}
\usepackage{amsmath,amsfonts,amsthm, amssymb}
\usepackage{fullpage}
\usepackage{authblk}
\usepackage{booktabs}
\usepackage{cite}
\usepackage{comment}
\usepackage{enumerate}
\usepackage[colorlinks=true]{hyperref}
\hypersetup{urlcolor=blue, citecolor=red}
\usepackage[utf8]{inputenc}
\usepackage{nicefrac}
\usepackage{paralist}
\usepackage{parskip}
\usepackage{tikz}
\usepackage{pgfplots,pgfplotstable}
\usepackage{pstricks}
\usepackage{psfrag}
\usepackage{dsfont}
\usepackage{cases}
 \usepackage{stmaryrd}
 \usepackage{graphics,float}

\newcommand{\email}[1]{\href{mailto:#1}{#1}}

\newcommand{\norm}[1]{\|#1\|}
\newcommand{\jump}[1]{\llbracket #1\rrbracket}
\newcommand{\avg}[1]{\lbrace\!\!\lbrace#1\rbrace\!\!\rbrace}
\newcommand{\bob}{{\boldsymbol \beta}}
\newcommand{\vv}{{\mathbf{v}}}
\newcommand{\vw}{{\mathbf{w}}}

\newcommand{\vx}{{\mathbf{x}}}
\newcommand{\vq}{{\mathbf{q}}}
\newcommand{\ben}{\begin{enumerate}}
\newcommand{\een}{\end{enumerate}}
\newcommand{\bea}{\begin{eqnarray}}
\newcommand{\eea}{\end{eqnarray}}

\newcommand{\vp}{{\mathbf{p}}}

\newcommand{\vW}{{\mathbf{W}}}

\newcommand{\cT}{{\mathcal T}}
\newcommand{\cQ}{{\mathcal Q}}
\newcommand{\cR}{{\mathcal R}}

\newcommand{\cF}{{\mathcal F}}

\newcommand{\be}{\begin{equation}}
\newcommand{\ee}{\end{equation}}
\newcommand{\eps}{\varepsilon}
\newcommand{\R}{\ensuremath{\mathbb{R}}}
\newcommand{\N}{\ensuremath{\mathbb{N}}}
\newcommand{\dsp}{\displaystyle}
\newcommand{\dt}{\partial_t}

\newcommand{\h}{{\mathfrak{h}}}
\newcommand{\Th}{\mathcal{T}_{\h}}
\newcommand{\Fh}{\mathcal{F}_{\h}}
\newcommand{\Fhi}{\Fh^{{\rm i}}}
\newcommand{\Fhb}{\Fh^{{\rm b}}}
\newcommand{\closure}[1]{\overline{#1}}
\newcommand{\normal}{\vec{n}}
\newcommand{\eqbydef}{\mathrel{\mathop:}=}
\newcommand{\st}{\; | \;}
\newcommand{\beq}{\begin{equation}}
\newcommand{\eeq}{\end{equation}}
\newcommand{\bM}{{\mathbb{M}}}
\newcommand{\bS}{{\mathbb{S}}}
\newcommand{\bE}{{\mathbb{E}}}
\newcommand{\bF}{{\mathbb{F}}}
\newcommand{\bD}{{\mathbb{D}}}
\DeclareMathOperator{\opcard}{card}
\newcommand{\card}[1]{\opcard(#1)}
\newenvironment{fig}{}{}

\newtheorem{theorem}{Theorem}

\newtheorem{remark}[theorem]{Remark}

\newtheorem{proposition}[theorem]{Proposition}

\makeatletter
\def\thm@space@setup{%
  \thm@preskip=\parskip \thm@postskip=0pt
}
\makeatother
\date{}
\title{A discontinuous Galerkin method for a new class of Green-Naghdi equations on simplicial unstructured meshes}
\author[1]{Arnaud Duran\footnote{\email{Arnaud.Duran@math.univ-toulouse.fr}}}
\affil[1]{INSA Toulouse, 135 Avenue de Rangueil, 31400 Toulouse, France}
\author[2]{Fabien Marche\footnote{\email{Fabien.Marche@math.univ-montp2.fr}}}
\affil[2]{University of Montpellier, Institut Montpelli\'{e}rain Alexander Grothendieck and Inria LEMON team, 34095 Montpellier, France}

\begin{document}

\maketitle

\begin{abstract}
In this paper, we introduce a discontinuous Finite Element formulation on simplicial unstructured meshes for the study of free surface flows based on the fully nonlinear and weakly dispersive Green-Naghdi equations. Working with a new class of asymptotically equivalent equations, which have a simplified analytical structure, we consider a decoupling strategy: we approximate the solutions of the classical shallow water equations supplemented with a source term globally accounting for the non-hydrostatic effects and we show that this source term can be computed through the resolution of scalar elliptic second-order sub-problems. The assets of the proposed discrete formulation are:
 \begin{inparaenum}[(i)]
  \item the handling of arbitrary unstructured simplicial meshes,
  \item an arbitrary order of approximation in space, 
  \item the exact preservation of the motionless steady states, 
  \item the preservation of the water height positivity, 
  \item a simple way to enhance any numerical code based on the nonlinear shallow water equations.
 \end{inparaenum}
   The resulting numerical model is validated through several benchmarks involving nonlinear wave transformations and run-up over complex topographies.
   \smallskip
  \\
  %\noindent\emph{2010 Mathematics Subject Classification:} 65N08, 65N30, 65N12
\end{abstract}

%% \tableofcontents

\section{Introduction}
\label{intro}
The mathematical modelling and numerical approximations of free surface water waves propagation and transformations in near-shore areas has received a lot of interest for the last decades, motivated by the perspective of acquiring a better understanding of important physical processes associated with the nonlinear and non-hydrostatic propagation over uneven bottoms. Great improvements have been obtained in the derivation and mathematical understanding of particular asymptotic models able to describe the behavior of the solution in some physical specific regimes. A recent review of the different models that may be derived can be found in \cite{Lannes:2009p4892}. In this work, we focus on the \textit{shallow water} and \textit{large amplitude} regime: the water depth $h_0$ is assumed to be small compared to the typical wave length $\lambda$:
$$
\mbox{(\textit{shallow water regime}) }\quad \mu:=\frac{h_0^2}{\lambda^2}\ll 1.
$$ 
while there is no assumption on the size of the wave's amplitude $a$:
\be
\mbox{(\textit{large amplitude regime}) }\quad \eps:=\frac{a}{h_0}=O(1).
\ee
Under this regime, the classical Nonlinear Shallow Water (NSW) equations can be derived from the full water waves equations by neglecting all the terms of order $O(\mu)$, see for instance \cite{lannes:book}. This model is able to provide an accurate description of important unsteady processes in the surf and swash zones, such as nonlinear wave transformations, run-up and flooding due to storm waves, see for instance \cite{Bonneton:2007p1191}, but it neglects the dispersive effects which are fundamental for the study of wave transformations in the shoaling area and possibly slightly deeper water areas.\\
Keeping the $O(\mu)$ terms in the analysis, the corresponding equations have been derived first by Serre \cite{Serre:1953p5261} in the horizontal surface dimension $d=1$, by Green and Naghdi \cite{CambridgeJournals:387109} for the $d=2$ case, and have
been recently mathematically justified in \cite{AlvarezSamaniego:2008p227}. \\
As far as numerical approximation is concerned, it is only recently that the Green-Naghdi (GN) equations have really received attention and several methods have been proposed including Finite Differences (FD), Finite Volumes (FV), Finite Elements (FE) or Spectral methods, see for instance among others \cite{Antunes-do-carmoSeabra-Santos:1993aa, Cienfuegos:2006p226, Li2014169, Dutykh:2013fk, shi:2012, PandaDawson:2014aa}.\\
After having developed and validated several hybrid FV-FD formulations based on a temporal splitting approach (see \cite{Bonneton20111479, MR2811693, rchgth89} for the $d=1$ case and \cite{lannes_marche:2014} for the $d=2$ case on structured cartesian grids), we have recently introduced in \cite{DuranMarche:2014ab} a new approach based on a fully discontinuous FE method in the $d=1$ case. This approach allows to decouple the hyperbolic and elliptic parts of the model by computing the solutions of the NSW equations supplemented by an additional algebraic source term, which fully accounts for the whole dispersive correction, and which is itself obtained from the resolution of auxiliary elliptic and coercive linear second order problems. With this approach we have paved the way towards:
 \begin{enumerate}
 \item[$\sharp$] more flexibility, as the proposed discrete formulation handles an arbitrary order of accuracy in space and, although limited in the $d=1$ case in \cite{DuranMarche:2014ab}, it conceptually generalizes to the dimension $d=2$ with unstructured meshes. It can moreover be straightforwardly used to enhance any numerical code based on the NSW equations, 
 
 \item[$\sharp$] more efficiency, as we built this approach on the new set of GN equations issued from \cite{lannes_marche:2014} which  allows to perform the corresponding matrix assembling and factorization in a preprocessing step, leading to considerable computational time savings.
 \end{enumerate}
Although similar decoupling strategies have subsequently been used to implement some hybrid approaches based on the classical GN equations (see \cite{Popinet:2015aa} for the $d=2$ case on cartesian meshes and \cite{FilippiniKazolea:2016aa} for the $d=1$ case), there is still no studies, up to our knowledge, that allows to approximate the solutions of GN equations on arbitrary unstructured meshes in the multidimensional case: this is the main purpose of this work. Starting from the new model recently introduced in \cite{lannes_marche:2014}, we consider a Discontinuous Galerkin (DG) formulation and generalize to the $d=2$ case the decoupling strategy of \cite{DuranMarche:2014ab}:
 \begin{enumerate}
 \item[$\sharp 1$] we compute the flow variables using a RK-DG approach \cite{Cockburn1998199} by approximating the solutions of the hyperbolic set of NSW equations supplemented with an additional source term that accounts for the whole dispersive correction, 
 \item[$\sharp 2$] the dispersive source term is obtained from the computation of auxiliary \textit{scalar} linear elliptic problems of second order, approximated through a mixed formulation and \textit{Local Discontinuous Galerkin (LDG)} stabilizing interface fluxes \cite{CockburnShu:1998aa, CastilloCockburn:2000aa}. 
 \end{enumerate}
 \noindent
 
 \noindent
This results in a global formulation of arbitrary order of accuracy in space, which is shown to preserve both the non-negativity of the water height and the motionless steady states up to the machine accuracy. These properties are particularly important in the study of propagating waves reaching the shore. Note that while this multidimensional strategy is clearly promising in terms of flexibility, the emphasize is put here on efficiency. First, the use of the \textit{diagonal constant} model issued from \cite{lannes_marche:2014} allows to considerably simplify the computations associated with the elliptic sub-problems. Indeed, the time evolutions of the velocity vector's components are here fully decoupled, thanks to the simplified analytical structure of the model. The problem therefore reduces to the approximation of \textit{scalar problems} without any third order derivatives, instead of the more complex vectorial ones obtained with the classical GN equations. The corresponding matrix can be assembled and the associated LU factorization stored in a pre-processing step. Secondly, the use of a \textit{nodal} approach \cite{HesthavenWarburton:2002aa} together with the \textit{pre-balanced} formulation of the hyperbolic part of the model (see for instance \cite{duran:2013, duran:marche:dg}), allows to combine 
\begin{inparaenum}[(i)]
 \item an efficient quadrature free treatment for the integrals which are not involved into the equilibrium balance, 
 \item a quadrature-based treatment with a lower computational cost, needed to exactly compute the surface and face integrals involved in the preservation of the steady states at rest,
 \item a direct nodal product method, in the spirit of pseudo-spectral methods, for the strongly nonlinear terms occurring in the source terms of the elliptic sub-problems.\end{inparaenum}\\
This remainder of this work is organized as follows: we describe the mathematical model and the notations in the next section. Section $3$ is devoted to the introduction of the discrete settings and the DG formulations for both hyperbolic and elliptic  sub-problems. This approach is then validated in the last section through convergence analysis and comparisons with data taken from experiments with several discriminating benchmark problems.

\section{The physical model}\label{model}

Let us denote by $\vx=(x,y)$ the horizontal variables, $z$ the vertical variable and $t$ the time variable. In the following, $\zeta(t,\vx)$ describes the free surface elevation with respect to its rest state, $h_0$ is a reference depth, $-h_0+b(\vx)$ is a parametrization of the bottom and $h:=h_0+\zeta-b$ is the water depth, as shown on Figure \ref{domain}. 
Denoting by $U_{hor}$ the horizontal component of the
velocity field in the fluid domain, we define the vertically averaged
horizontal velocity $\vv=(u,v)\in \R^2$ as
$$
\vv(t,\vx)=\frac{1}{h}\int_{-h_0+b}^\zeta U_{hor}(t,\vx,z)dz,
$$ 
and we denote by $\vq=h\vv$ the corresponding horizontal momentum.
\medskip
\begin{figure}[H]
\psfrag{z}{{\footnotesize $z$}}
\psfrag{0}{{\footnotesize $0$}}
\psfrag{h0}[c][c]{\textcolor{white}{7\footnotesize{2}}{\footnotesize $-h_0$}}
\psfrag{N}{{\footnotesize $z=\zeta(t,\vx)$}}
\psfrag{F}{{\footnotesize $z=-h_0+b(\vx)$}}
\psfrag{X}{{\footnotesize $\vx=(x,y)$}}
\hspace{0.45em}\includegraphics[width=0.75\textwidth]{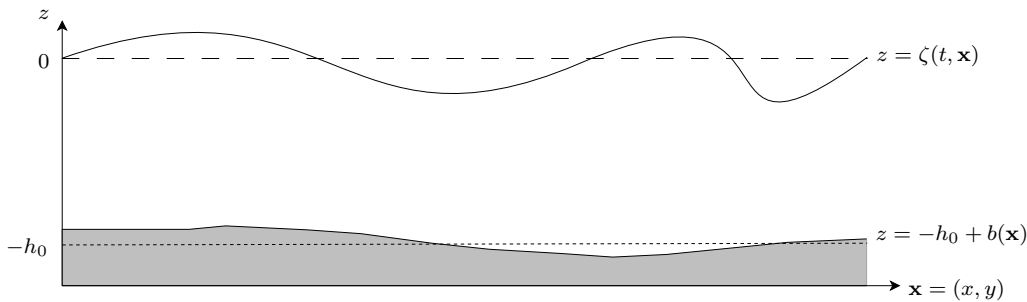}
\caption{Sketch of the domain}
\label{domain}
\end{figure}

\subsection{The Green-Naghdi equations}\label{NDGN}

According to \cite{Bonneton20111479}, the Green-Naghdi equations may be written as follows:
\be\label{eq6}
	\left\lbrace
	\begin{array}{l}
	\dsp \dt \zeta +\nabla\cdot (h \vv)=0,\\
        \dsp (I+\cT[h,b])\big[\dt \vv+(\vv\cdot \nabla)\vv\big]+
	g\nabla\zeta 
	+\cQ_1[h,b](\vv)=0,
	\end{array}\right.
\ee
where the linear operator $\cT[h,b]\cdot$ and the quadratic form $\cQ_1[h,b](\cdot)$ 
are defined for all smooth enough $\R^2$-valued function $\vw$ by
\bea
	\label{eq2}
		\cT[h,b]\vw &=&
	\cR_1[h,b](\nabla\cdot \vw)
	+\cR_2[h,b](\nabla b\cdot \vw),\\
	\label{eq11}
	\cQ_1[h,b](\vw) &=& -2\cR_1[h,b](\partial_1\vw\cdot\partial_2\vw^\perp+(\nabla\cdot \vw)^2)
	+\cR_2[h,b](\vw\cdot(\vw\cdot\nabla)\nabla b),
\eea
(here $\partial_1$ and $\partial_2$ denote space derivatives in the
two horizontal directions) with, for all smooth enough scalar-valued function $w$,
\bea
	\label{eq4}
	\cR_1[h,b]w&=&-\frac{1}{3h}\nabla(h^3 w)
	-\frac{h}{2}w\nabla b,\\
	\label{eq5}
	\cR_2[h,b]w&=&\frac{1}{2h}\nabla(h^2 w)
	+w\nabla b.
\eea
We also recall in \cite{Bonneton20111479} that the dispersion properties of (\ref{eq6}) can be improved by adding some terms of order $O(\mu^2)$ to the momentum equation, which consequently  does not affect the accuracy of the model. An asymptotically equivalent enhanced family of models parametrized by $\alpha > 0$ is given by
\be\label{eq6imp}
	\left\lbrace
	\begin{array}{l}
	\dsp \dt \zeta +\nabla\cdot (h \vv)=0,\\
        \dsp (I+\alpha\cT[h,b])\left( \dt \vv+(\vv\cdot \nabla)\vv+\frac{\alpha-1}{\alpha}g\nabla\zeta\right)
        +\frac{1}{\alpha}g\nabla\zeta 
	+\cQ_1[h,b](\vv)=0.
	\end{array}\right.
\ee
or alternatively written in $(h,h\vv)$ variables:
\be\label{eq6imphhVbis}
	\left\lbrace
	\begin{array}{l}
	\dsp \dt h +\nabla\cdot (h \vv)=0,\\
       	\dsp (I+\alpha {\bf T}[h,b])\left( \dt (h\vv) 
        \!+\!\nabla\!\cdot\! (h\vv \otimes \vv) + \!\frac{\alpha\!-\!1}{\alpha} gh \nabla \zeta\right) + \frac{1}{\alpha} gh\nabla\zeta 
	+ h\cQ_1[h,b](\vv)  = 0,
	\end{array}\right.
\ee
where we have introduced the operator ${\bf T}[h,b]$ defined as follows
$$
{\bf T}[h,b]\vw=h\cT[h,b](\frac{\vw}{h}).
$$
As shown in \cite{DuranMarche:2014ab}, \eqref{eq6imphhVbis} can be recast as
\begin{subnumcases}{}
     \dt h +\nabla\cdot (h \vv)=0,\label{disp:1}   \\
    \dt (h\vv) +\nabla\cdot(h\vv\otimes \vv) + gh \nabla \zeta
+ \mathcal{D}_o=0,\label{disp:2}\\
(I\!+\!\alpha {\bf T}[h,b])(\mathcal{D}_o + \frac{1}{\alpha}gh \nabla \zeta ) = \frac{1}{\alpha} gh\nabla\zeta 
	+ h\cQ_1[h,b](\vv).\label{disp:3}
    \end{subnumcases}
which highlights the fact that the dispersive correction of order $O(\mu)$ only acts as a source term $\mathcal{D}_o$ in \eqref{disp:2}, and is obtained as the solution of an auxiliary second-order elliptic sub-problem \eqref{disp:3}.

\begin{remark}\label{remark:benefits}
These formulations have two main advantages:
\begin{enumerate}
\item They do not require the computation of third order derivatives, while this is necessary in the standard formulation of the GN equations, 

%Indeed third order derivatives can be factored out by $I+{\bf T}[h,b]$ (or $I+\cT[h,b]$ if one prefers to work with $(\zeta, V)$) as highlighted in the equivalent form
%$$
%\dt (hV) 
%        +\nabla\cdot(hV\otimes V) + \frac{\alpha-1}{\alpha} gh \nabla \zeta
%+(I\!+\!\alpha {\bf T}[h,b])^{-1}\big[\frac{1}{\alpha} h\nabla\zeta 
%	+ h\cQ_1[h,b](V)  \big]=0.
%$$ 
\item The presence of the operator $(I+\alpha{\bf T}[h,b])^{-1}$ makes the models robust
with respect to high frequency perturbations, which is an interesting property for numerical computations. 
\end{enumerate}
\end{remark}

\begin{remark}
These formulations have two main drawbacks: 
\begin{enumerate}
\item Solving linear systems arising from discrete formulations of \eqref{eq6imphhVbis} may be computationally expensive as $I+\alpha{\bf T}[h,b]$ is a \emph{matricial} second order differential operator acting on two-dimensional vectors. This structure entails a coupling of the time evolutions of the two components of $h\vv$ through \eqref{disp:2}.

\item $I+\alpha{\bf T}[h,b]$ is a time dependent operator (through the dependence on $h$) and this is of course the same for any associated discrete formulations: the corresponding matrices have to be assembled at each time step or sub-steps.
\end{enumerate}
\end{remark}

\subsection{A Green-Naghdi model with a simplified analytical structure}\label{model:diag:const}

In \cite{lannes_marche:2014} , some new families of models are introduced to overcome these drawbacks, without loosing the benefits listed in Remark \ref{remark:benefits}. In particular, defining a modified water depth at rest (which therefore does not depend on time):
$$h_b=\max(h_0-b, \eps_0)=\max(h-\zeta,\eps_0),$$ where $\eps_0$ is a numerical threshold introduced to account for possible dry areas in a consistent way, the  \textit{one-parameter optimized constant-diagonal} GN equations read as follows:
\be\label{eq6impnewter}
	\left\lbrace
	\begin{array}{l}
	\dsp \dt h +\nabla\cdot (h \vv)=0,\\
	
	\vspace{0.1cm}
       	\dsp \big[1+\alpha\mathds{T}[h_b]\big]\left(\dt (h\vv) 
        +\nabla\cdot (h\vv\otimes \vv)+\frac{\alpha-1}{\alpha}gh\nabla\zeta\right) +
	 \frac{1}{\alpha}h\nabla\zeta + \mathds{Q}[h,b](\zeta, \vv)=0,
	 \
%\hspace{4.5cm}
%         +h(\cQ_1(\vv) 
%	 +g\cQ_2(\zeta))
%        +g\cQ_3(\zeta)=0
	\end{array}\right.
\ee
where for all smooth enough $\R$-valued function $w$:
\be\label{Tcg}
\mathds{T}[h_b]w= -\frac{1}{3}\nabla\cdot(h_b^3\nabla\frac{w}{h_b}),
\ee
and
\be\label{Qcg}
\mathds{Q}[h,b](\zeta, \vv)=h(\cQ_1[h,b](\vv) +g\cQ_2[h,b](\zeta)) +g\cQ_3[h,h_b]\left( \big[1+\alpha \mathds{T}[h_b]\big]^{-1}(h\nabla\zeta)\right),
\ee
is a second order nonlinear operator with
\be\label{Q2CG}
\cQ_2[h,b](\zeta) = -h(\nabla^\perp h\cdot\nabla) \nabla^\perp\zeta
-\frac{1}{2h}\nabla(h^2 \nabla b\cdot \nabla\zeta) + \big(\frac{h}{2}\Delta\zeta - (\nabla b\cdot \nabla\zeta)\big)\nabla b,
\ee
and for all smooth enough $\R$-valued function $w$
%\begin{equation}\label{defQ3}
%\cQ_3(\zeta)=-\cS[h^2-h_b^2]\big[I+\mu {\bf T}_{diag}^b\big]^{-1}(h\nabla\zeta).
%\end{equation}
\be\label{defS}
\cQ_3[h,h_b] w= \frac{1}{6}\nabla (h^2-h_b^2)\cdot \nabla w +\frac{h^2-h_b^2}{3}\Delta w -\frac{1}{6}\Delta(h^2-h_b^2)w.
\ee

\begin{remark}
We actually show in \cite{lannes_marche:2014} that it is indeed possible to replace the inversion of $I+\alpha{\bf T}[h,b]$ by the inversion of $1+\alpha\mathds{T}[h_b]$, where $\mathds{T}[h_b]$ depends only on the fluid at rest (i.e. $\zeta=0$), while keeping the asymptotic $O(\mu^2)$ order of the expansion. The interest of working with (\ref{eq6impnewter}) rather than (\ref{eq6imphhVbis}) is that $\mathds{T}[h_b]$ has a simplified \textit{scalar} structure, \textit{i.e.} it can be written in matricial form as
\begin{equation}\label{structTtilde}
\left(\begin{array}{cc}
-\frac{1}{3}\nabla\cdot(h_b^3\nabla\frac{1}{h_b}\cdot) & 0\\
0  & -\frac{1}{3}\nabla\cdot(h_b^3\nabla\frac{1}{h_b}\cdot)
\end{array}\right).
\end{equation}
From a numerical viewpoint, this simplified analytical structure allows to compute each component of the discharge $h\vv$ separately, and to alleviate the computational cost associated with the dispersive correction of the model, as the discrete version of $1+\alpha\mathds{T}[h_b]$ may be assembled and factorized once and for all, in a preprocessing step.
\end{remark}
\begin{remark}
 The other difference with (\ref{eq6imphhVbis}) is the presence of the modified quadratic term $\mathds{Q}[h,b]$, which shares with (\ref{eq6imphhVbis}) the nice property that no computation of third order derivative is needed. The price to pay is the inversion of an extra linear
system, through the computation of $\cQ_3[h,h_b]\left( \big[1+\alpha \mathds{T}[h_b]\big]^{-1}(h\nabla\zeta)\right)$. However, this extra cost is largely off-set by the gain obtained by using the time independent scalar operator $\mathds{T}[h_b]$. 
\end{remark}

\begin{remark}\label{reminv}
One could replace $\cQ_3[h,h_b]\left( \big[I+\alpha\mathds{T}[h_b]\big]^{-1}(h\nabla\zeta)\right)$ by $\cQ_3[h,h_b]\left(h\nabla\zeta\right)$ in the second equation of (\ref{eq6impnewter}), keeping the same asymptotic $O(\mu^2)$ order. This would avoid the resolution of this extra linear system but leads to
strong instabilities. We
refer to \cite{lannes_marche:2014}  for more comments on this point.
\end{remark}

\begin{remark}
As shown in \cite{DuranMarche:2014ab}, the model \eqref{eq6impnewter} can be recast as follows:
\begin{subnumcases}{\left(\mathcal{CG}_\alpha\right)}
     \dt h +\nabla\cdot (h \vv)=0,\label{disp:1c}   \\
    \dt (h\vv) +\nabla\cdot(h\vv\otimes \vv) + gh \nabla \zeta
+ \mathcal{D}_c=0,\label{disp:2c}\\
\big[1\!+\!\alpha\mathds{T}[h_b] \big](\mathcal{D}_c + \frac{1}{\alpha} h\nabla\zeta)= 
	 h( \frac{1}{\alpha} \nabla\zeta + \cQ_1[h,b](\vv) +g\cQ_2[h,b](\zeta)) +\cQ_3[h,h_b]\mathcal{K},\label{disp:3c}\\
\big[1\!+\!\alpha\mathds{T}[h_b]\big]\mathcal{K} =	gh\nabla\zeta,\label{disp:4c}
    \end{subnumcases}    
and we see that the dispersive correction $\mathcal{D}_c$ acting as a source term in \eqref{disp:2c}, is obtained as the solution of auxiliary scalar second-order elliptic sub-problems. 
% 
%Additionally, when breaking waves are expected, one need to recover the shallow water equations. To achieve this, simple reformulations may be obtained
%$$
%    \dt (h\vv) +\nabla\cdot(h\vv\otimes \vv) + gh \nabla \zeta
%+ \tilde{\mathcal{D}}_c=0,
%$$
%with $\tilde{\mathcal{D}}_c = \mathcal{D}_c-\frac{1}{\alpha}gh \nabla\zeta$.
\end{remark}

\begin{remark}
We highlight that $\left(\mathcal{CG}_\alpha\right)$ still does not involve third order derivative computation, which is a very interesting property. It is indeed shown in \cite{FilippiniKazolea:2016aa} that when third order derivatives on the free surface occur, it is important to introduce some sophisticated approximation strategies for the free surface gradient to reduce the dispersion error. 
\end{remark}

\subsection{A pre-balanced Green-Naghdi formulation}\label{model:pre:bal}

Paving the way towards the construction of a well-balanced and efficient discrete formulation in \S\ref{discrete:sec}, we now adjust the \textit{pre-balanced} approach of \cite{Liang2009873} for the case $d=1$ and \cite{duran:2013} for the case $d=2$ to the Green-Naghdi equations $\left(\mathcal{CG}_\alpha\right)$. We introduce the {\it total free surface elevation} $\eta=h+b$, denote $\vW={}^t(\eta, \vq)$ and use the following splitting of the hydrostatic pressure term: 
\be\label{pre:bal:split}
gh\nabla \zeta = \frac{1}{2}g\nabla(\eta^2 - 2\eta b) + g\eta \nabla b,
\ee
to obtain the \textit{pre-balanced} formulation of \eqref{disp:1c}-\eqref{disp:2c}
%\begin{subnumcases}{}
%     \dt \eta +\nabla\cdot \vq=0,\label{disp:1d}   \\
%    \dt \vq +\nabla\cdot(\frac{\vq\otimes \vq}{\eta-b}) + \frac{1}{2}g\nabla(\eta^2 - 2\eta b)
%+ \tilde{\mathcal{D}}_c= - g\eta \nabla b,\label{disp:2d}
% \end{subnumcases}
 \begin{subnumcases}{}
     \dt \eta +\nabla\cdot \vq=0,\label{disp:1d}   \\
    \dt \vq +\nabla\cdot \mathcal{F}(\vW, b)  
+ \mathcal{D}_c= \mathcal{B}(\eta,b),\label{disp:2d}%- g\eta \nabla b,\label{disp:2d}
 \end{subnumcases}     
 with the flux and topography source terms defined as
\be\label{flux:pre:bal}
\mathcal{F}(\vW,b) =  \frac{\vq\otimes \vq}{\eta-b} + \frac{1}{2}g\nabla(\eta^2 - 2\eta b) I_2,\quad\quad
 \mathcal{B}(\eta,b)=- g\eta \nabla b.
\ee
We obtain a \textit{one-parameter pre-balanced constant-diagonal Green-Naghdi} model (more simply referred to as $\left(\mathcal{PCG}_\alpha\right)$ model in the following):
\begin{subnumcases}{}
     \dt \eta +\nabla\cdot \vq=0,\label{pb:disp:1c}   \\
      \dt \vq +\nabla\cdot \mathcal{F}(\vW, b)  
+ \mathcal{D}_c= \mathcal{B}(\eta,b),\label{pb:disp:2c}\\
\big[1\!+\!\alpha\mathds{T}[h_b] \big](\mathcal{D}_c + \frac{1}{\alpha} gh\nabla\eta)= 
	h(\frac{1}{\alpha}g\nabla\eta + \cQ_1[h,b](\vv) +g\cQ_2[h,b](\eta)) +\cQ_3[h,h_b]\mathcal{K},\label{pb:disp:3c}\\
\big[1\!+\!\alpha\mathds{T}[h_b]\big]\mathcal{K} =	gh\nabla\eta.\label{pb:disp:4c}
    \end{subnumcases}    

\section{Discrete formulation}\label{discrete:sec}
\subsection{Settings and notations}
Let $\Omega\subset\R^d$, $d=2$, denote an open bounded connected polygonal domain with boundary $\partial\Omega$. We consider a geometrically conforming mesh $\Th$ defined as a finite collection of nonempty disjoint open triangular elements $T$ of boundary $\partial T$ such that $\closure{\Omega}=\bigcup_{T\in\Th}\closure{T}$. The meshsize is defined as $\h=\max\limits_{T\in\Th} \,\h_T$
with $\h_T$ standing for the diameter of the element $T$ and we denote $\vert T \vert$ the area of $T$, $\mathfrak{p}_T$ its perimeter and $\normal_{T}$ its unit outward normal.\\
\noindent
Mesh faces are collected in the set $\Fh$ and the length of a face $F\in\Fh$ is denoted by $\vert F\vert $. A mesh face $F$ is such that either there exist $T_1,T_2\in\Th $ such that $F\subset\partial T_1\cap\partial T_2$ ($F$ is called an interface and $F\in\Fhi$) or there exists $T\in\Th$ such that $F\subset\partial T \cap\partial\Omega$ ($F$ is called a boundary face and $F\in\Fhb$). 
For all $T\in\Th$, $\cF_T\eqbydef\{F\in\Fh\st F\subset\partial T \}$ denotes the set of faces belonging to $\partial T $ and, for all $F\in\cF_T$, $\normal_{TF}$ is the unit normal to $F$ pointing out of $T$.\\
\noindent
In what follows, we consider $\mathbb{P}^{k}(\Th)$ the broken bivariate polynomial space defined as follows:
\begin{equation}\label{broken}
\mathbb{P}^{k}(\Th):= \lbrace{ v \in L^{2}(\Omega) \, \vert \,  \;v_{\vert T}  \,\in \, \mathbb{P}^{k}(T) \;\;\forall \, T \, \in \,\mathcal{T}_{\h}\rbrace}, 
\end{equation}
where $\mathbb{P}^{k}(T)$ denotes the space of bivariates polynomials in $T$ of degree at most $k$, and 
we define $X_\h = \mathbb{P}^{k}(\Th)\times \left( \mathbb{P}^{k}(\Th)\right)^2$. We denote $N_k = \dim(\mathbb{P}^{k}(T))=(k+1)(k+2)/2$.\\

\noindent
To discretize in time, for a given final computational time $t_{\textrm{max}}$, we consider a partition $(t^n)_{0\leq n\leq N}$ of the time interval $[0, t_{\textrm{max}}]$ with $t^0=0$, $t^N=t_{\textrm{max}}$ and $t^{n+1}-t^{n}=\Delta t^n$. For any sufficiently regular function of time $w$, we denote by $w^n$ its value at discrete time $t^n$. More details on the computations of $\Delta t^n$ and the time marching algorithms are given in \S\ref{Time}.
%------------------------------------------------------------------------------%

\subsection{Discrete formulation for the advection-dominated equations}

Postponing to \S\ref{sect:disp} the computation of the projection of the dispersive correction $\mathbb{D}(\vW,b)$ on the approximation space,  we focus here on the discrete formulation associated to \eqref{pb:disp:1c}-\eqref{pb:disp:2c}, written in a more compact way:
 \be\label{compact}
 \partial_t \vW+\nabla\cdot \mathbb{F}(\vW,b) + \mathbb{D}(\vW,b)=\mathbb{B}(\vW,b),
 \ee
with
\begin{align*}
\mathbb{B}(\vW,b) = {}^t \left( 0, {}^t\mathcal{B}(\eta,b) \right), \;\;\;\;  \mathbb{D}(\vW,b)={}^t\left(0, {}^t\mathcal{D}_c\right)
\;\;\;\;\mbox{and}\;\;\;\;
\mathbb{F}(\vW,b) = 
\begin{pmatrix} {}^t\vq \\  \mathcal{F}(\vW,b) \end{pmatrix}.
\end{align*}

\subsubsection{The discrete problem}
We seek an approximate solution $\vW_\h={}^t\left(\eta_\h, {}^t\vq_\h \right)$ of \eqref{compact} in $X_\h$. Requiring the associated residual to be orthogonal to $\mathbb{P}^{k}(\Th)$, the semi-discrete formulation reduces to the local statement: find $(\eta_\h, {}^t\vq_\h)$ in $X_\h$ such that
\begin{align}
&\int_{T} \frac{d}{dt} \vW_\h\pi_{\h}d\vx -\int_{T}\mathbb{F}(\vW_\h, b_\h)\cdot\nabla \pi_{\h}d\vx \,+ \int_{\partial T}(\mathbb{F}(\vW_\h, b_\h)\cdot \normal_{T}) \,\pi_{\h}ds\nonumber\\
&\quad\quad\quad\quad\quad + \int_{T} \mathbb{D}(\vW_\h,b_\h)\pi_{\h}d\vx = \int_{T} \mathbb{B}(\vW_\h,b_\h)\pi_{\h}d\vx,\label{discrete:2}
\end{align}
for all $\pi_\h\in\mathbb{P}^{k}(\Th)$ and all element $T\in\Th$, where $b_\h$ refers to the $L^2$ projection of $b$ on $\mathbb{P}^{k}(\Th)$ and $\mathbb{D}(\vW_\h,b_\h)$ stands for a polynomial description of $\mathbb{D}(\vW,b)$ in $X_\h$, to be obtained in \S\ref{sect:disp}.

\begin{remark}\label{rem:base:T}
Considering a local basis $\{ \phi_i\}_{i=1}^{N_k}$ for a given element $T\in \Th$,
the local discrete solution $\vW_{\h\vert T}$ may be expanded as:
\begin{equation}\label{expansion}
\vW_{\h\vert T}(\vx,t) = \sum\limits_{i=1}^{N_k} \tilde{\vW}_i(t) \phi_i(\vx),\;\;\forall \vx\in T, \forall t \in [0, t_{max}],
\end{equation}
where $\{\tilde{\vW}_i(t)\}_{i=1}^{N_k}$ are the local expansion coefficients, defined as $\tilde{\vW}_i(t)={}^t(\tilde{\eta}_i, {}^t\tilde{\vq}_i)$ and $\tilde{\vq}_i = {}^t((\tilde{q}_x)_i, (\tilde{q}_y)_i)$.\\
 Many choices are of course possible for the basis. In the following, $\{ \phi_i\}_{i=1}^{N_k}$ refers to the interpolant (nodal) basis on the element $T$ and we choose the Fekete nodes \cite{TaylorWingate:2000aa} as approximation points.\\
 \end{remark}
\noindent
Equipped with such local basis, following Remark \ref{rem:base:T}, and splitting the boundary integral into faces integrals, the local statement \eqref{discrete:2} is now equivalent to
% \begin{equation}
% \begin{split}
%\sum\limits_{i=1}^{N_k}& \mathbf{M}_{ij} \frac{d}{ dt}\tilde{\vW}_i(t) 
%  -\int_{T}\mathbb{F}(\vW_{\h},b_\h)\cdot\nabla \phi_j(\vx)d\vx\\& + \int_{\partial T} (\mathbb{F}(\vW_{\h},b_\h)\cdot \normal_{T})\,\phi_j\,ds + \int_{T}\mathbb{D}(\vw_{\h},b_\h)\phi_j\,d\vx= \int_{T}\mathbb{B}(\vw_{\h},b_\h)\phi_j\,d\vx,\;\;\\& \hspace{6cm} 1\leq j\leq N_d.
%\label{Weak_formulation2}
%\end{split}
%\end{equation}
\begin{equation}
 \begin{split}
\sum\limits_{i=1}^{N_k}& \left(\int_{T} \phi_{i}\phi_{j}\,d\vx\right) \frac{d}{ dt}\tilde{\vW}_i(t) 
  -\int_{T}\mathbb{F}(\vW_{\h},b_\h)\cdot\nabla \phi_j\,d\vx +    \sum\limits_{F\in\cF_T} \int_{F} \widehat{\mathbb{F}}_{TF}\,\phi_j\,ds\\&   + \int_{T}\mathbb{D}(\vW_{\h},b_\h)\phi_j\,d\vx= \int_{T}\mathbb{B}(\vW_{\h},b_\h)\phi_j\,d\vx,\;\;\quad  1\leq j\leq N_k,
\label{Weak_formulation2}
\end{split}
\end{equation}
where $ \widehat{\mathbb{F}}_{TF}$ is a stabilizing numerical approximation of the normal interface flux $\mathbb{F}(\vW_{\h},b_\h)\cdot \normal_{TF}$, to be defined in the following.

\subsubsection{Interface fluxes and well-balancing}\label{numerical:flux}
We recall in the following a simple choice to approximate the interface fluxes \cite{duran:marche:dg}, leading to a well-balanced scheme that preserves motionless steady states. This modified flux can also be seen as the adaptation of the ideas of \cite{xing:2013} to the {\it pre-balanced} formulation (\ref{disp:1d})-(\ref{disp:2d}).\\
Consider a face $F\in\cF_T$ (for the sake of simplicity, we only focus on the case $F\in \Fh^i$ and do not detail the case $F\in \Fh^b$). Let us denote denote $\vW^{-}$ and $\vW^{+}$ respectively the {\it interior} and {\it exterior} traces on $F$, with respect to the elements $T$. Similarly, $b^{-}$ and $b^{+}$ stand for the {\it interior} and {\it exterior} traces of $b_{\h}$ on $F$.
We define:
\beq\label{topo:check}
b^* = \max(b^-,b^+),\quad\quad \check{b} = b^* - \max(0, b^* - \eta^-)
\eeq
and
\begin{align}
&\check{h}^- = \max(0, \eta^- - b^*), \quad\quad \check{h}^+ = \max(0, \eta^+ - b^*),\label{def:1}\\
&\check{\eta}^- = \check{h}^- + \check{b}, \quad\quad\quad\quad\quad\;\; \check{\eta}^+ = \check{h}^+ + \check{b},\label{def:2}
\end{align}
leading to the new {\it interior} and {\it exterior} values:
\beq\label{ext_int}
\check{\vW}^{-} = {}^t(\check{\eta}^-,  \frac{\check{h}^-}{\eta^- - b^-}{\bf q^-}),\;\;\;\; \check{\vW}^{+} = {}^t(\check{\eta}^+,  \frac{\check{h}^+}{\eta^+ - b^+}{\bf q^+}).
\eeq
Now we set
\beq\label{num:flux}
\widehat{\mathbb{F}}_{TF} = \mathbb{F}_\h(\check{\vW}^{-}, \check{\vW}^{+}, \check{b}, \check{b},\normal_{TF}) + \widetilde{\mathbb{F}}_{TF},
\eeq
as the numerical flux function through the interface $F$,
where:
\begin{enumerate}
\item the numerical flux function $\mathbb{F}_\h$ is the global Lax-Friedrichs flux:
\beq\label{LFflux}
\mathbb{F}_\h(\vW^-, \vW^+, b^-, b^+,\normal_{TF}) = \frac{1}{2}( \mathbb{F}(\vW^-, b^-)\cdot \normal_{TF} +  \mathbb{F}(\vW^+, b^+)\cdot \normal_{TF}  - a(\vW^+-\vW^-)),
\eeq
with $a= \max\limits_{T\in\Th} \lambda_T$ and
\beq\label{lambda}
\lambda_T =  \max\limits_{\partial T} \left( \left | \frac{\vq_{\h\vert T}}{\eta_{\h\vert T} - b_{\h\vert T}} \cdot \normal_{T} \right | +\sqrt{g(\eta_{\h\vert T} - b_{\h\vert T})} \right).
\eeq
\item $\widetilde{\mathbb{F}}_{TF}$ is a correction term defined as follows:
\begin{equation}
\widetilde{\mathbb{F}}_{TF} = \begin{pmatrix} 0 & 0 \\ g\check{\eta}^-(\check{b} - b^-) & 0 \\0 & g\check{\eta}^-(\check{b} -b^-)\end{pmatrix}\cdot \normal_{TF}.
\end{equation}
\end{enumerate}
\noindent
Note that the modified interface flux (\ref{num:flux}) only induces perturbations of order $k+1$ when compared to the traditional interface fluxes. 
\begin{remark}\label{rem:LF}
Let denote in the following by $w_{T}$ the averaged value of the discrete approximation $w_{\h\vert T}$ on the element $T$, for any $\R$ or $\R^2$-valued function $w$.  Let consider a first order scheme for the averaged free-surface:
\begin{equation}\label{schemaDGave2}
\eta^{n+1}_{T} =  \eta^{n}_{T} - \frac{\Delta t^n}{\vert T\vert}\sum\limits_{F\in\cF_T} \int_{F} \mathbb{F}_\h(\check{\vW}^{-}, \check{\vW}^{+}, \check{b}, \check{b},\normal_{TF}) \,ds,
\end{equation}
with $\mathbb{F}_\h$ defined following (\ref{LFflux}) and $\check{\vW}^{-}$, $\check{\vW}^{+}, \check{b}$ obtained from (\ref{topo:check})-(\ref{ext_int}), starting from $\vW^{-}$ and $\vW^{+}$ defined respectively, for each face $F$, as the first-order \textit{piecewise constant}  interior and exterior values $\vW_{T_1}^n$ and $\vW_{T_2}^n$, with $\cT_F=\{T_1, T_2\}$.\\
Then, assuming that $h^{n}_{T}\geq 0,\;\forall T\in \cT_\h$, we have $h^{n+1}_{T}\geq 0$ $\forall T\in \cT_\h$, under the condition
\beq\label{cfl:firstorder}
\lambda_T  \frac{ \mathfrak{p}_T}{\vert T\vert}  \Delta t^n \leq 1,\;\;\forall T\in \cT_\h.
\eeq
This is a positivity property of the \textit{Lax-Friedrichs} flux extended to the \textit{pre-balanced} formulation, which is shown in \cite{duran:marche:dg} and is mandatory to obtain a positive high-order DG scheme, see \cite{Zhang20103091}.
\end{remark}

\subsubsection{Positivity}\label{dry}

The enforcement strategy of the water height non-negativity within the DG formulation \cite{xing:2013} requires positivity of the water height at carefully chosen quadrature nodes at the beginning of each time step. Then, the positivity of the water height is ensured providing the use of a positivity preserving first order scheme, like (\ref{LFflux}) (see Remark \ref{rem:LF}) and a suitable time-step restriction.\\
We adjust these ideas to the $\left(\mathcal{PCG}_\alpha\right)$ formulation (\ref{pb:disp:1c})-(\ref{pb:disp:2c}) to enforce the mandatory property that the mean value of the water height $h_{T}=\eta_T-b_T$ on any given element $T\in\Th$ remains positive during the time marching procedure. The main ideas are summarized for an explicit first order \textit{Euler} scheme in time for the sake of simplicity:
\begin{enumerate}
\item[$\sharp$ 1]  considering the broken polynomial space $\mathbb{P}^{k}(T)$, we assume that the face integrals are computed using $(k+1)$-points \textit{Gauss} quadrature (see Remark \ref{quad:pre}). The special quadrature rules introduced in \cite{Zhang:2012fk} are obtained by a transformation of the tensor product of a $\beta$-points \textit{Gauss-Lobatto} quadrature (with $\beta$ the smallest integer such that $2\beta-3\geq k$) and the $(k+1)$-point \textit{Gauss} quadrature. This new quadrature includes all $(k+1)$-point Gauss quadrature nodes for each face $F\in \cF_T$, involves positive weights and it is exact for the integration of $\eta^{n}_{\h}$ over $T$. In the following, let us denote $S_{T}^k$ the set of points of this quadrature rule. We show on Figure \ref{quadraturepoints} the resulting quadrature nodes used for $k=2$ and $k=3$ orders of approximation on a reference element. 

\item[$\sharp$ 2]  for each element $T\in\Th$,  $h_{\h\vert T}^n$ is computed from $\eta_{\h\vert T}^n$ and $b_{\h\vert T}$ and we need to ensure that $h_{\h\vert T}^n(\vx) \geq 0,\; \forall \vx\in S_{T}^k$, which is a sufficient condition to ensure the non-negativity property for the DG scheme (\ref{discrete:2}), under the CFL-like condition \eqref{cfl2}.\\
This condition is enforced using the accuracy preserving limiter of \cite{Zhang20103091}. Assuming $h_{T}^{n}\geq 0$, we replace $h_{\h\vert T}^n$ by a conservative linear scaling around this element average:
\beq\label{limiter:robuste}
\hat{h}_{\h\vert T}^n = \theta_{T}(h_{\h\vert T}^n - h_{T}^{n}) + h_{T}^{n},
\eeq
with
$$
\theta_{T} = \min \left( \frac{h_{T}^{n}}{h_{T}^{n} -m_{T}}  ,1\right), \;\;\;\;\mbox{and}\;\;\;\;
m_{T} = \min\limits_{\vx\in S_{T}^k} \;h_{\h\vert T}^n(\vx).
$$
%\item[$\sharp$ 3] if the first order scheme ...
\end{enumerate}
\begin{remark}
Note that this approach also ensures that the water height remains positive at the $k+1$ \textit{Gauss} quadrature nodes used to compute the faces integrals in \eqref{Weak_formulation2}.
\end{remark}

\begin{figure}
\begin{center}
\includegraphics[scale=0.22,angle=-90]{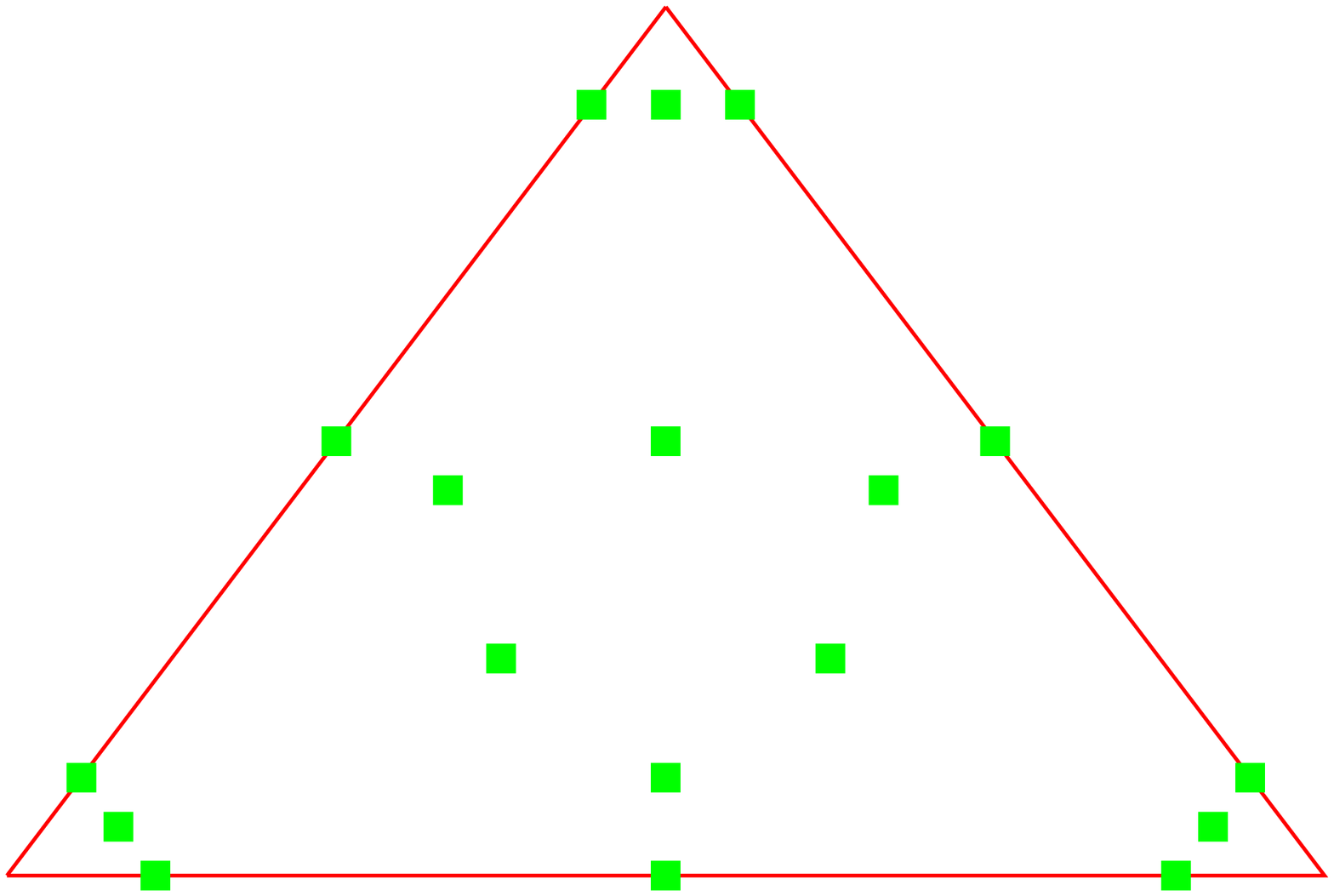}
\includegraphics[scale=0.22,angle=-90]{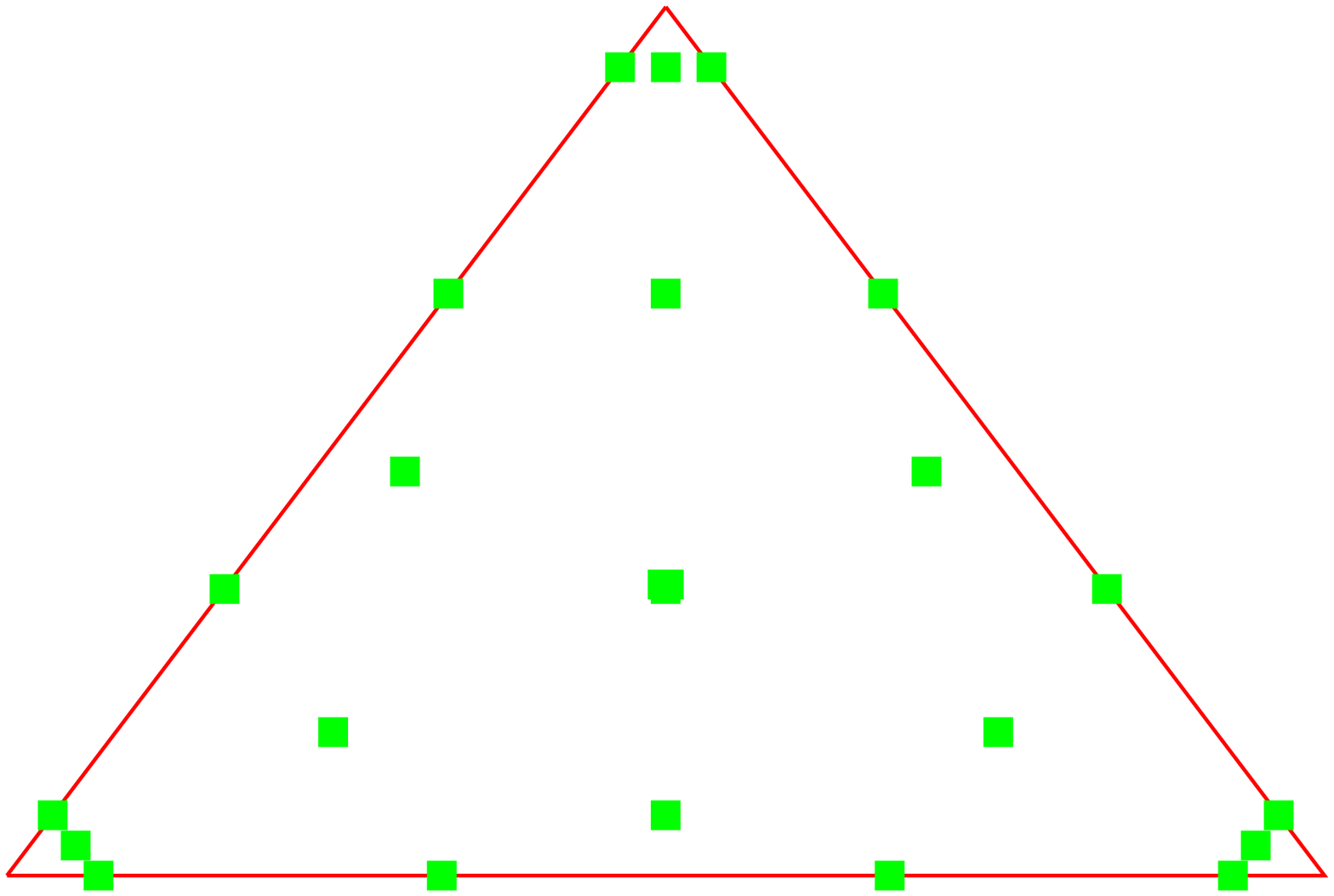}
\caption{Nodes locations for the special \textit{Zhang and Shu} quadrature - $\mathbb{P}^2$ and $\mathbb{P}^3$ cases.}
\label{quadraturepoints}
\end{center}
\end{figure}

\subsection{Discrete formulations for the second order elliptic sub-problems}\label{sect:disp}

We are now left with the computation of the projection $\mathbb{D}(\vW_{\h},b_\h)$ of the dispersive correction on the approximation space $X_\h$. As its first component identically vanishes, the computation of the discrete version of $\mathcal{D}_c$ is obtained as the solution of the discrete problems associated with
\begin{subnumcases}{}
\big[1\!+\!\alpha\mathds{T}[h_b]\big]\mathcal{K} =	gh\nabla\eta,\label{subell1}\\
\big[1\!+\!\alpha\mathds{T}[h_b] \big](\mathcal{D}_c + \frac{1}{\alpha} gh\nabla\eta)= h(\frac{1}{\alpha} g\nabla\eta + \cQ_1[h,b](\vv) +g\cQ_2[h,b](\eta)) +\cQ_3[h, h_b]\mathcal{K}.\label{subell2}
    \end{subnumcases}    
    
    %\subsubsection{The $1\!+\!\alpha\mathds{T}[h_b]$ operator}
Note that both problems ultimately reduce to the construction of a discrete formulation associated with the generic scalar problem 
$$\big[1\!+\!\alpha\mathds{T}[h_b]\big] w = f,$$ 
where the source term $f$ is successively defined as the directional scalar components of the source terms occurring in \eqref{subell1}-\eqref{subell2}, that is to say
\begin{equation}\label{source:term:f}
f=
\begin{cases}
gh\nabla\eta\;\; \mbox{for}\; \eqref{subell1}\\
\displaystyle\frac{1}{\alpha} gh\nabla\eta + h(\cQ_1(\vv) +g\cQ_2(\eta)) +\cQ_3[h, h_b]\mathcal{K}\;\;\mbox{for}\; \eqref{subell2}, 
\end{cases}
\end{equation}
where we keep a vectorial notation for the sake of simplicity.
Remarking now that the following identity holds for any smooth enough scalar-valued function:
\begin{equation}\label{struct:div}
\mathds{T}[h_b]w = -\frac{1}{3}\nabla\cdot\left(H^b\nabla w \right) + \frac{1}{6}\nabla\cdot\left( w\nabla H^b\right),
\end{equation}
with the notation $H^b=h_b^2$, we consider the following mixed formulation, introducing a diffusive flux $\vp$:
\be\label{mixed:disp}
\begin{cases}
w-\alpha\nabla\cdot\left(\frac{1}{3} H^b \vp - \frac{1}{6}w\nabla H^b\right) = f,\\
\vp - \nabla w = 0,
\end{cases}
\ee
and the following associated discrete problem: find $(w_\h, {}^t\vp_\h)$ in $X_\h$ such that
\begin{align}
&\int_{T} w_\h\pi_{\h}d\vx + \frac{\alpha}{3}\int_{T} H^b_\h (\vp_\h\cdot\nabla \pi_{\h})d\vx - \frac{\alpha}{3}\int_{\partial T} H^b_\h (\vp_\h\cdot \normal_T)\pi_\h\,ds\,\label{discrete:1T}\\&\hspace{2cm}-\frac{\alpha}{6}\int_T w_\h(\nabla H^b_\h\cdot\nabla\pi_\h)\,d\vx +\frac{\alpha}{6}\int_{\partial T} w_\h(\nabla H^b_\h\cdot\normal_T)\pi_\h\,ds = \int_T f_\h\pi_\h\,d\vx,\nonumber\\
&\int_{T} \vp_\h\cdot\phi_{\h}d\vx + \int_T w_\h\nabla\cdot\phi_\h\,d\vx - \int_{\partial T} w_\h(\phi_\h\cdot\normal_T)\,ds= 0,\label{discrete:2T}
\end{align}
for all $(\pi_\h,\phi_\h)\in X_\h$, where $H^b_\h$ and $\nabla H^b_\h$ respectively denote the $L^2$ projections of $h_b^2$ and $\nabla h_b^2$ on $\mathbb{P}^{k}(\Th)$ and $(\mathbb{P}^{k}(\Th))^2$, and $f_\h$ is a polynomial representation of $f$ in $\mathbb{P}^{k}(\Th)$, with $f$ defined according to \eqref{source:term:f}. We choose to use the \textit{Local Discontinuous Galerkin (LDG)} approach \cite{CockburnShu:1998aa} to compute the associated stabilizing fluxes. For any given element $T\in\Th$, we have: 
%\begin{equation}
\begin{align}
&\int_{\partial T} H^b_\h (\vp_\h\cdot \normal_T)\pi_\h\,ds=\sum\limits_{F\in\cF_T}\int_{F} H^b_\h (\widehat{\vp}_{TF}\cdot \normal_{TF})\pi_h\,ds,\\
&\int_{\partial T} w_\h(\phi_\h\cdot\normal_T)\,ds = \sum\limits_{F\in\cF_T}\int_{F}\widehat{w}_{TF}(\phi_\h\cdot\normal_{TF})\,ds,\\
&\int_{\partial T} w_\h(\nabla H^b_\h\cdot\normal_T)\pi_\h\,ds= \sum\limits_{F\in\cF_T}\int_{F} \widehat{w}_{TF} (\nabla H^b_\h\cdot\normal_T)\pi_\h\,ds,
\end{align}
where the interface fluxes are computed as follows:
\begin{align}
%\begin{cases}
&\widehat{w}_{TF} = \avg{w_\h} + (\bob\cdot\normal_{TF})\jump{w_\h},\label{flux:ldg:1}\\
&\widehat{\vp}_{TF} = \avg{\vp_\h} - \bob(\jump{\vp_\h}\cdot\normal_{TF}) + \frac{\xi}{\vert F\vert}\jump{w_\h}\normal_{TF},\label{flux:ldg:2}
%\end{cases}
\end{align}
with the classical notations for the face average $\avg{w} =(w^{+}+w^{-})/2$ and face jump $\jump{w}= w^{+}-w^{-}$, and $w^{-}$ and $w^{+}$ stand respectively for the interior and exterior traces of $w_{\h}$ with respect to the considered face $F$ (and similar notations for the $\R^2$-valued quantity $\vp$). The penalization parameter $\xi$ and the upwinding parameter $\bob$ are both set to $1$.
\begin{remark}
The choice of the \textit{LDG} fluxes (\ref{flux:ldg:1})-(\ref{flux:ldg:2}) allows to eliminate locally the auxiliary discrete flux $\vp_\h$, and to globally assemble the matrix corresponding to the discrete formulation of $\big[1\!+\!\alpha\mathds{T}[h_b]\big]$. Note also that, for the sake of simplicity,  we also use the upwind flux (\ref{flux:ldg:1}) to approximate the faces contributions associated with the purely advective part of (\ref{struct:div}). 
\end{remark}
\begin{remark}
The relative simplicity of the discrete formulation (\ref{discrete:1T})-(\ref{discrete:2T}) is a direct consequence of the simplified analytical structure of the $\left(\mathcal{PCG}_\alpha\right)$ formulation. One can of course adapt this approach to the initial formulation (\ref{eq6imphhVbis}), leading in practice to a more costly and complicated global assembling process, as the operator $\cT[h,b]$ directly acts on a $\R^2$-valued function.
\end{remark}

%\subsubsection{The $\mathds{Q}[h,b]$ operator}\label{sec:Qop}
For the sake of efficiency, the computations of the source terms integrals $\int_T f_\h\pi_\h\,d\vx$ occurring in \eqref{discrete:1T}  are performed in a collocation way, in the spirit of pseudo-spectral methods. More precisely, we build approximated integrands with polynomial representations $(\cQ_{j,\h})_{j=1,2,3}$ in $\left( \mathbb{P}^{k}(\Th)\right)^2$ of the operators $(\cQ_j)_{j=1,2,3}$ using direct products of the discrete flow variables $(\eta_\h, \vv_\h)$ and their derivatives at the \textit{Fekete} nodes. These derivatives are however weakly computed with the DG approach, using LDG fluxes. In practice, this is simply achieved by adjusting the approach of \cite{DuranMarche:2014ab} to the $d=2$ case to easily compute the required first and second order derivatives of $\eta_\h$ and $\vv_\h$ in each space direction. Again, the second order spatial derivatives are written in mixed form  and we use the stabilizing fluxes \eqref{flux:ldg:1}-\eqref{flux:ldg:2}. The corresponding discrete formulations are simplified through local elimination of the diffusive fluxes, leading to the assembling of global matrices for first and second order weak derivatives in each direction. For instance, considering derivatives in the first direction, the mixed form
\begin{equation*}
v + \partial_x w = 0 \quad , \quad u + \partial_x v = 0,\, 
\end{equation*}
leads to the following local discrete formulations for $1\leq j\leq N_k$:
%\bea
%\label{LDG1} \int_{T} v \pi_{\h} d\vx = \int_{T} w \partial_x \pi_{\h} d\vx - \int_{\partial T} \widehat{w}  \hat{n}_{x}^l\, \pi_{\h}\,ds \, , \\
%\label{LDG2} \int_{T} u \pi_{\h} d\vx = \int_{T} v \partial_x \pi_{\h} d\vx - \int_{\partial T} \widehat{v}  \hat{n}_{x}^l\, \pi_{\h}\,ds, 
%\eea
\begin{equation}\label{LDG_scheme}
\begin{split}
\sum\limits_{i=1}^{N_k} \tilde{v}_i  M_{ij} & =
\sum\limits_{i=1}^{N_k} \tilde{w}_i S_{ij}^x -\sum\limits_{F\in\cF_T} \int_{F}  \widehat{w}_{TF} \phi_{j} n_{TF}^x ds\, , \\
 \sum\limits_{i=1}^{N_k} \tilde{u}_i M_{ij} & =
\sum\limits_{i=1}^{N_k} \tilde{v}_i S_{ij}^x -\sum\limits_{F\in\cF_T} \int_{F} \widehat{v}_{TF} \phi_{j} n_{TF}^x\, ds,
 \end{split}
\end{equation}
where $n_{TF}^x$ is the first component of $\normal_{TF}$ and with
\begin{equation*}
M_{ij} := \int_{T} \phi_{i} \phi_{j}\,d\vx \quad , \quad S^x_{ij} := \int_{T} \phi_{i} \partial_x \phi_{j}\, d\vx.
\end{equation*}
These systems can be globally rewritten as follows:
\begin{equation}\label{LDG_Matrice}
\begin{split}
&\bM \tilde{V} = \bS_{x} \tilde{W} - \left(\bE_{x} - \bF_{x} \right) \tilde{W} \, , \\
&\bM \tilde{U} = \bS_{x} \tilde{V} - \left(\bE_{x} + \bF_{x} \right) \tilde{V} - \frac{\xi}{\h} \bF_{x}  \tilde{W},
\end{split}
 \end{equation} 
where: 
 \begin{inparaenum}[(i)]
 \item $\tilde{U}, \tilde{V}$ and $\tilde{W}$ are $N_k\times \card \Th$ vectors gathering the  expansion coefficients $\tilde{u}, \tilde{v}$ and $\tilde{w}$ for all mesh elements, 
 \item $\bM$ and $\bS_x$ are the square $\card{\Th} \times N_k$ global mass and stiffness matrices, with a block-diagonal structure,
 \item $\bE_x$ and $\bF_x$ are square $\card{\Th} \times N_k$ matrices, globally assembled by gathering all the mesh faces contributions, accounting respectively for the average $\avg{\cdot}$ and jump $\jump{\cdot}$ operators occurring in the definition \eqref{flux:ldg:1}-\eqref{flux:ldg:2} of the \textit{LDG} fluxes.
 \end{inparaenum} Note that each interface $F\in\cF^i_\h$ contributes to four blocks of size $N_k$ in the global matrices and each boundary face $F\in\cF^b_\h$ contributes to one block. Equipped with these global structures, first and second order global differentiation matrices can be straightforwardly assembled, leading to the following compact notations:
%\begin{equation*}
%\tilde{V} = \bD_x \tilde{W} \quad  , \quad \tilde{U} = \bD_{xx} \tilde{W} \, ,
%\end{equation*}
%with 
\begin{align}
&\tilde{V} = \bD_x \tilde{W},\;\;\mbox{with}\;\;\bD_x = \bM^{-1} \left(\bS_x - \bE_x + \bF_x \right),\\
&\tilde{U} = \bD_{xx} \tilde{W},\;\;\mbox{with}\;\; \bD_{xx} = \bM^{-1} \Big(\left(\bS_x-\bE_x-\bF_x\right)\bD_x - \frac{\xi}{\h} \bF_x \Big).
\end{align}
Similar constructions are performed for $\bD_y$ and $\bD_{yy}$. 

\begin{remark}
The use of such a direct nodal products method helps to reduce the computational cost by reducing the accuracy of quadrature but may threaten the numerical stability for strongly nonlinear or marginally resolved problems with the possible introduction of what is known in the field of spectral methods as aliasing driven instabilities. Several well-known methods may help to alleviate this issue, like the use of stabilization filtering methods. In this work however, and considering the test cases studied in \S\ref{section:numerical}, we did not need to use any additional stabilization mechanism. 
%unless for the convergence stuwe use a mild nodal filter following the approach of \cite{Engsig-Karup:2006fk} which is enough to remove the possibly observed instabilities without damaging the convergence.
\end{remark}

\subsection{Time-marching, boundary conditions and wave-breaking}

\subsubsection{Time discretization}\label{Time}
The time stepping is carried out using the explicit third-order SSP-RK scheme \cite{gottliedtadmor}. Up to $k=3$, we consider RK-SSP schemes of order $k+1$. A fourth order SSP-RK scheme is used for $k\geq 3$. For instance, writing the semi-discrete equations as $\frac{d}{dt}\vW_{\h}+\mathcal{A}_{\h}(\vW_{\h}) = 0$, advancing from time level $n$ to $n + 1$ is computed as follows with the third-order scheme:
\begin{equation} \label{RK3}
\left \lbrace \begin{array}{ll}
\vspace{0.2cm}
\vW_{\h}^{n,1}=\vW_{\h}^{n}- \Delta t^n \tilde{\mathcal{A}}_{\h}(\vW_{\h}^{n}) \, , \\
\vspace{0.2cm}
\vW_{\h}^{n,2}=\frac{1}{4}(3\vW_{\h}^{n}+\vW_{\h}^{n,1})- \frac{1}{4}\Delta t^n \tilde{\mathcal{A}}_{\h}(\vW_{\h}^{n,1}) \, ,\\
\vW_{\h}^{n+1}=\frac{1}{3}(\vW_{\h}^{n}+2\vw_{\h}^{n,2})- \frac{2}{3}\Delta t^n \tilde{\mathcal{A}}_{\h}(\vW_{\h}^{n,2}) \, .
 \end{array} \right. 
\end{equation}
with the formal notation $\tilde{\mathcal{A}}_\h = \mathcal{A}_\h \circ \Lambda\Pi_\h$, $\vW_{\h} \leftarrow \Lambda\Pi_\h \vW_{\h}$ being the limitation operator possibly acting on the approximated vector solution  (see \S\ref{Limitation}), and $\Delta t^n$ is obtained from the CFL condition \eqref{cfl2}.

\subsubsection{Boundary conditions}\label{boundary}

For the test cases studied in the next section, boundary conditions are imposed weakly, by enforcing suitable reflecting relations at virtual exterior nodes, at each boundaries, allowing to compute the corresponding interface stabilizing fluxes.  Solid-wall (reflective) but also periodic conditions (as the computational domains geometries of the cases studied in \S\ref{section:numerical} are rectangular) can be enforced following this simple process. \\
These simple boundary conditions possibly have to be complemented with \textit{ad-hoc} absorbing boundary conditions,  allowing the dissipation of the incoming waves energy together with an efficient damping of possibly non-physical reflections, and generating boundary conditions that mimic a wave generator of free surface waves. We have implemented relaxation techniques and we enforce periodic waves combined with generation/absorption by mean of a generation/relaxation zone, following the ideas of \cite{madsen2003}, using the relaxation functions described in \cite{Wei:1995p885}, and the computational domain is locally extended to include sponge layers which may also include a generating layer. The length of these layers has to be calibrated from the incoming waves (generally 2 or 3 wavelengths).

\subsubsection{Wave-breaking and limiting}\label{Limitation}
Obviously, vertically averaged models cannot reproduce the surface wave overturning and are therefore inherently unable to fully model wave breaking. Moreover, if the GN equations can accurately reproduce most phenomena exhibited by non-breaking waves in finite depth, including the steepening process occurring just before breaking, they do not intrinsically account for the energy dissipation mechanism associated with the conversion into turbulent kinetic energy observed during broken waves propagation. \\
Several methods have been proposed to embed wave breaking in depth averaged models.  Many of them focus on the inclusion of an energy dissipation mechanism through the activation of extra terms in the governing equations when wave breaking is likely to occur, and the reader is referred to \cite{tissier2} for a recent review of these approaches. More recently, hybrid strategies have been elaborated for weakly nonlinear models \cite{Tonelli:2009p1244, Kazolea:2014kx} and for fully nonlinear models \cite{Bonneton20111479, tissier2, tissier1, shi:2012}. Roughly speaking, the idea is to switch from GN to NSW equations when the wave is ready to break by locally suppressing the dispersive correction and the various approaches may differ by the level of sophistication of the detection criteria and the switching strategies.\\ 
Denoting that, from a numerical point of view, such an approach allows to avoid numerical instabilities by turning off the computation of higher-order derivatives and nonlinear and non-conservative terms in the vicinity of appearing singularities, we use in \cite{DuranMarche:2014ab} a purely numerical smoothness detector to identify the potential instability areas, switch to NSW equations in such areas by simply locally neglecting the $\mathcal{D}_c$ source term and use a limiter strategy to stabilize the computation, letting breaking fronts propagate as moving bores.\\
As the aim of the present work is only to introduce and validate our fully discontinuous discrete formulation in the multidimensional case, we do not focus on the development of new wave breaking strategies and we simply adjust this simple approach to the $d=2$ case to possibly stabilize the computations performed in \S\ref{section:numerical}. \\
We detect the troubled elements (following the terminology of \cite{Qiu:siam:2005}) using the criterion proposed in \cite{Krivodonova2004323}, and based on the strong superconvergence phenomena exhibited at element's outflow boundaries. More precisely, for any $T\in\cT_\h$, denoting by $\partial T_{in}$ the inflow part of $\partial T$, we use the following criterion:
\begin{equation}\label{sensor}
\mathbb{I}_T= \frac{\sum\limits_{F\in\partial T_{in}}\displaystyle\int_{F}(h^- - h^+)\,ds}{\h_T^{(k+1)/2}\vert \partial T_{in}\vert  \| h_{\h\vert T}\| },
\end{equation}
which has already successfully been used as a \textit{troubled cells} detector in purely hyperbolic shallow water models, see among others \cite{Ern:2008p3536, kesser2012, zhu:2013}.\\
If $\mathbb{I}_T\geq 1$ then we apply a slope limiter on each scalar component of $\vW_{\h\vert T}$, based on the \textit{maxmod} function, see \cite{Burbeau2001111}. This limiting strategy is not recalled here, as we straightforwardly reproduce the implementation described in \cite{duran:marche:dg} for the NSW equations in \textit{pre-balanced} form. Of course, waves about to break do not embed any free surface singularities yet, but our numerical investigations have shown that when the wave has steepened enough, the free surface gradient becomes large enough to activate the criteria. This simple approach, although quite rough and far less sophisticated than recent strategies introduced for instance in \cite{tissier2, Kazolea:2014kx}, allows us to obtain good results in the various  cases of \S\ref{section:numerical}.

\subsection{Main properties}
\noindent
We have the following result:

\begin{proposition} \label{well_balancing}
The discrete formulation (\ref{Weak_formulation2}) together with the interface fluxes discretization (\ref{num:flux}) and a first order \textit{Euler} time-marching algorithm has the following properties:
\begin{enumerate}
\item it preserves the motionless steady states, providing that the integrals of \eqref{Weak_formulation2} are exactly computed for the motionless steady states. In other terms, we have for all $n\in\N$:
\begin{equation} \label{WB} \Big( \left \lbrace \begin{array}{ll}
\eta_{\h}^{n} \equiv \eta^e \\
\textbf{q}_{\h}^{n} \equiv  0 \end{array} \right. \Big) \Rightarrow \quad \Big(\left \lbrace \begin{array}{ll}
\eta_{\h}^{n+1} \equiv \eta^e \\
\textbf{q}_{\h}^{n+1} \equiv 0 \end{array} \right. \Big),
\end{equation}
with $\eta^e$ constant,
\item assuming moreover that $h_{T}^{n}  \geq  0,\;\forall T\in\cT_\h$  and $h_{\h\vert T}^n(\vx)\geq 0,\;\forall \vx\in S_{T}^k,\;\forall T\in\cT_\h$, then we have $h_{T}^{n+1}  \geq 0, \;\forall T\in\cT_\h$
under the condition
\beq\label{cfl2}
\lambda_T  \frac{ \mathfrak{p}_T}{\vert T\vert} \Delta t^n \leq \frac{2}{3} \hat{\omega}_1^\beta,
\eeq
where $\hat{\omega}_1^\beta$ is the first quadrature weight of the $\beta$-point \textit{Gauss-Lobatto} rule used in \S\ref{dry}.
\end{enumerate}
\end{proposition}
\begin{proof}
\begin{enumerate}
\item assuming that the following equilibrium$\vW_\h=\vW_\h^e=(\eta^e, 0)$ holds, we have to show that $\forall T\in\Th$ and $1\leq j\leq N_k$:
\begin{equation}
\int_{T}\mathbb{F}(\vW_{\h}^e,b_\h)\cdot\nabla \phi_j\,d\vx -    \sum\limits_{F\in\cF_T} \int_{F} \mathbb{F}_{TF}^e\,\phi_j\,ds  - \int_{T}\mathbb{D}(\vW_{\h}^e,b_\h)\phi_j\,d\vx+\int_{T}\mathbb{B}(\vW_{\h}^e,b_\h)\phi_j\,d\vx=0,\\
\label{wb:1}
\end{equation}
where $\mathbb{F}_{TF}^e$ is the interface numerical flux obtained at equilibrium.
Looking at (\ref{num:flux}), and highlighting that for each interface $F$ we have $\check{\eta}^- = \check{\eta}^+=\eta^e$ and therefore $\check{\vW}^{-}=\check{\vW}^{+}$, it is easy to check that $ \mathbb{F}_{TF}^e= \mathbb{F}(\vW^-, b^-)\cdot\normal_{TF}$, thanks to the consistency of the numerical flux function $\mathbb{F}_\h$. Consequently, we have
\begin{align}\label{wb:4}
\int_{T}\mathbb{F}(\vW_{\h}^e,b_\h)\cdot\nabla \phi_j d\vx -    \sum\limits_{F\in\cF_T} \int_{F} \mathbb{F}_{TF}^e\,\phi_j\,ds &= -\int_{T}\nabla\cdot\mathbb{F}(\vW_{\h}^e,b_\h)\phi_jd\vx,\\
&=-\int_{T} \mathbb{B}(\vW_{\h}^e,b_\h)\phi_jd\vx.
\end{align}
assuming that the integrals are computed exactly (see Remark \ref{quad:pre}), and
observing that we have $\nabla_\h\cdot \mathbb{F}(\vW_{\h}^e,b_\h)=\mathbb{B}(\vW_{\h}^e,b_\h)$. For the integrals associated with the dispersive source term, it is straightforward to check from the definitions \eqref{eq11} and \eqref{Q2CG} of $(\cQ_j)_{j=1,2}$ that $\cQ_{1,\h}=0$ whenever $\vq_\h=0$ and that $\cQ_{2,\h}=0$ whenever $\eta_\h=\eta^e$. Moreover, using again the fact that $\nabla_\h\eta^e=0$, the polynomial representation $\mathcal{K}_\h$ of $\mathcal{K}$, obtained as the solution of the discrete problem \eqref{discrete:1T}-\eqref{discrete:2T} associated with \eqref{subell1}, taking $f_\h=0$ as source term and suitable boundary conditions, identically vanishes leading to $\cQ_{3,\h}=0$ using \eqref{defS}. Note that this result holds no matter what method we use to compute polynomials representations $(\cQ_{j,\h})_{j=1,2,3}$. 

%we first highlight that  $\nabla_\h\eta^e=0$ and denoting by $(\cQ_{j,\h}^e)_{j=1,2,3}$ the polynomial approximations of $(\cQ_j)_{j=1,2,3}$ obtained at equilibrium, we observe that  $\cQ_{2,\h}^e=0$ and that the projection of $\mathcal{K}$ on $( \mathbb{P}^{k}(\Th))^2$, obtained as the solution of the discrete problem obtained with \eqref{subell1}, identically vanishes, leading to $\cQ_{3,\h}^e=0$. In a similar way, it is easy to check that $\cQ_{1,\h}^e=0$ leading to $\mathcal{D}_{c,\h}=0$, obtained as the solution of the discrete problem associated with \eqref{subell2} with a zero second member. 
%let us denote $(\cQ_{j,\h}^e)_{j=1,2,3}$ the polynomial approximations of $(\cQ_j)_{j=1,2,3}$ obtained at equilibrium.
%\textcolor{red}{ that no matter what method we use}
\item the proof of \cite{Zhang:2012fk}, adjusted to the pre-balanced framework in\cite{duran:marche:dg}, can be straightforwardly reproduced, as the dispersive source term $\mathbb{D}(\vW_{\h},b_\h)$ has no influence on the mass conservation equation.
Note that such positivity preservation can be extended to the third order SSP scheme of \S\ref{Time}, see \cite{Zhang:2012fk}.
\end{enumerate}
\end{proof}
\begin{remark}\label{quad:pre}
The face integrals of \eqref{Weak_formulation2} can be \textit{exactly} computed at motionless steady states with a $k+1$ point Gauss quadrature rule thanks to the \textit{pre-balanced} splitting \eqref{pre:bal:split}. Indeed, with the resulting formulation of the advective flux \eqref{flux:pre:bal}, the integrands on faces are of order $2k$ at equilibrium. In the same way, the cubature rules used for the surface integrals have to be exact only for bivariate polynomials of order $2k-1$. 
\end{remark}

%==================================================
%==================================================
%  Numerical validations
%==================================================
%==================================================

\section{Numerical validations}\label{section:numerical}
In this section, we validate the previous discrete formulation through several benchmarks. Unless stated otherwise, we use periodic boundary conditions in each direction, we set $\alpha=1.159$, $\eps_0=0.1$ and the time step restriction is computed according to \eqref{cfl2}. Accordingly with the non-linear stability result of the previous section, we do not suppress the dispersive effects in the vicinity of dry areas. Some accuracy analysis are performed in the first two cases using the broken $L^2$ norm defined as follows for any arbitrary scalar valued piecewise polynomial function $w_\h$ defined on $\Th$:
$$
\norm{w_\h}_{L^2(\Th)} = \left( \sum\limits_{T\in\Th}\norm{w_{\h\vert T}}_{L^2(T)}^2 \right)^{\frac{1}{2}}.
$$

\subsection{Preservation of motionless steady state}
This preliminary test case is devoted to check the ability of the formulation to preserve motionless steady states and accuracy validation. The computational domain is the [-1,1] $\times$ [-1,1] square, and we use an unstructured mesh of $8466$ elements. The bottom elevation involves a bump and a hollow having same dimensions, respectively located at $\vx_1=(x_1,y_1)=(-\dfrac{1}{3} , -\dfrac{1}{3})$ and  $\vx_2=(x_2,y_2)=(\dfrac{1}{3} , \dfrac{1}{3})$, leading to the following analytic profile :
\begin{equation}
b(r_1,r_2)= 1 + d \, e^{-(r_{1}/L)^{2}} - d \, e^{-(r_{2}/L)^{2}} \, ,
\end{equation}
where $r_{1,2}$ are respectively the distances from $\vx_1$ and $\vx_2$ and we set $d = 0.45$ and $L = 0.15$. The reference water depth is $h_{0} = 1.5\,m$, leading to the configuration depicted on Fig. \ref{WB_fig1}. Our numerical investigations confirm that this initial condition is preserved up to the machine accuracy for any value of polynomial order $k$. For instance, the $L^{2}$ numerical errors obtained at $t=50\,s$ using a $k=2$ approximation are respectively $3.0e$-$16$, $7.4e$-$16$ and $7.9e$-$16$ for $\eta$, $hu$ and $hv$.  \\
Keeping the same computational domain and topography profile, we also perform a convergence study, using a reference solution obtained with $k=4$ and a regular  triangulation with space steps $\Delta x = \Delta y= 2^{-9}\,m$. The initial free surface is set to:
\begin{equation}
\zeta(t=0, \vx) = a \,e^{-(\norm{\vx}/L)^2} \, \quad , \quad a = 0.075\,h_0.
\end{equation}
The computations are performed on a sequence of regular triangular meshes with increasing refinement ranging from $2^{-4}\,m$ to $2^{-9}\,m$ and polynomial expansions of degrees ranging from $k=1$ to $k=4$, while keeping the time step constant and small enough to ensure that the leading error orders are provided by the spatial discretization. The $L^{2}$ errors computed at $t_{\textrm{max}}=0.02\,s$ lead to the convergence curves shown on Fig. \ref{order_h} for the free surface elevation and Fig. \ref{order_q} for the discharge. The corresponding convergence rates obtained by linear regression are reported on each curve. We note that while the $L^2$-errors are typically larger on the discharge than on the free surface, we asymptotically reach some convergence rates between $\mathcal{O}(\h^{k+\frac{1}{2}})$ and $\mathcal{O}(\h^{k+1})$ for both variables. We observe that for this problem, the resulting convergence rates are greater than the $\mathcal{O}(\h^{k+\frac{1}{2}})$ optimal estimate one would be expected, see \cite{ZhangShu:2004aa}.

\begin{figure}
\begin{center}
%\begin{fig}
\includegraphics[scale=0.25,angle=0]{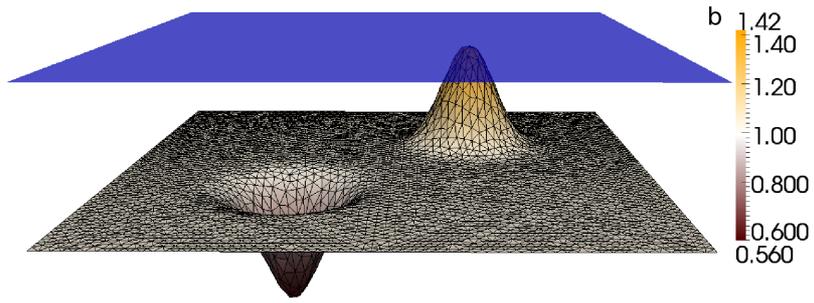} 
%\end{fig}
\caption{Test 1 - Motionless steady states preservation: topography and free surface}
\label{WB_fig1}
\end{center}
%\end{minipage}
\end{figure}

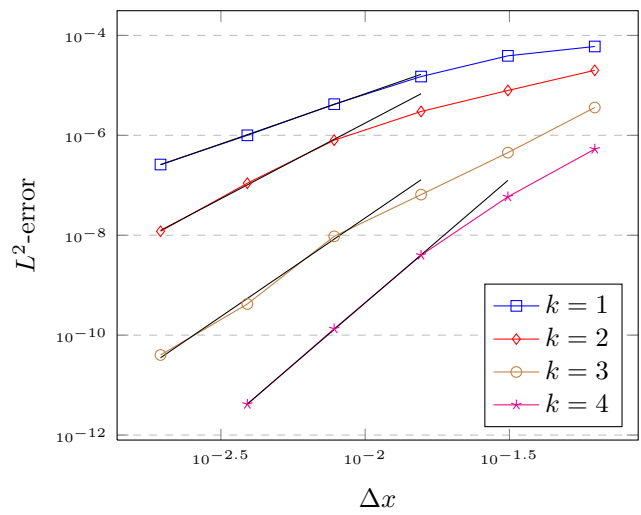
\begin{figure}
%\begin{minipage}{0.5\linewidth}
\centering
\begin{tikzpicture}
\begin{loglogaxis}[
    xlabel={$\Delta x$},	
    ylabel={$L^2$-error},
    xtick={0,-1,-2,-3,-4},
    legend pos=south east,
    ymajorgrids=true,
    grid style=dashed,
    yticklabel style = {font=\tiny},
    xticklabel style = {font=\tiny},
    %width=1\textwidth,
    %height=0.9\textwidth,
]
\addplot[
    color=blue,
    mark=square,
    ] table [y=a, x expr=1/\thisrow{dx}] {conv_h1.txt};
\addplot[
    color=red,
    mark=diamond,
    ] table [y=b, x expr=1/\thisrow{dx}] {conv_h1.txt};
\addplot[
    color=brown,
    mark=o,
    ] table [y=c, x expr=1/\thisrow{dx}] {conv_h1.txt};
\addplot[
    color=magenta,
    mark=star,
    ] table [y=d, x expr=1/\thisrow{dx}] {conv_h4.txt};
    
 \addplot[color=black,
	mark=no,]
	 table[ 
	 x expr=1/\thisrow{dx},
         y={create col/linear regression={y=a,
         variance list={10, 10, 10, 1000}}}]{conv_h1_line.txt}
         coordinate [pos=0.25] (A)
         coordinate [pos=0.1]  (B);
% save the slope parameter:
\xdef\slope{\pgfplotstableregressiona}
% draw the opposite and adjacent sides
% of the triangle
\draw (A) -| (B)
     node [pos=0.1,anchor=south]
     {\pgfmathprintnumber{\slope}};
      
\addplot[color=black,
	mark=no,]
	table[
         x expr=1/\thisrow{dx},
         y={create col/linear regression={y=b,
         variance list={10, 10, 10,1000}}}]{conv_h1_line.txt}
         coordinate [pos=0.1] (C)
         coordinate [pos=0.25]  (D)
   ;
      % save the slope parameter:
\xdef\slopeq{\pgfplotstableregressiona}
      % draw the opposite and adjacent sides
% of the triangle
\draw (C) -| (D)
     node [pos=0.2,anchor=north]
     {\pgfmathprintnumber{\slopeq}};
     %%%%%
     
\addplot[color=black, mark=no,]
	 table[
         x expr=1/\thisrow{dx},
         y={create col/linear regression={y=c,
         variance list={10, 10, 10, 1000}}}]{conv_h1_line.txt}
         coordinate [pos=0.25](E)
         coordinate [pos=0.1](F);
        % save the slope parameter:
        \xdef\slopec{\pgfplotstableregressiona}
        % draw the opposite and adjacent sides
        % of the triangle
        \draw (E) -| (F)
        node [pos=0.2,anchor=south]
       {\pgfmathprintnumber{\slopec}};
     
 \addplot[color=black, mark=no,]
	 table[
         x expr=1/\thisrow{dx},
         y={create col/linear regression={y=d,
         variance list={10, 10, 10, 1000}}}]{conv_h4_line.txt}
         coordinate [pos=0.1](G)
         coordinate [pos=0.23](H);
        % save the slope parameter:
        \xdef\sloped{\pgfplotstableregressiona}
        % draw the opposite and adjacent sides
        % of the triangle
        \draw (G) -| (H)
        node [pos=0.2,anchor=north]
       {\pgfmathprintnumber{\sloped}};

%\draw[-stealth] (A) -| (B);
    \legend{$k=1$,$k=2$,$k=3$,$k=4$}
\end{loglogaxis}

\end{tikzpicture}
\caption{Test 1 - $L^2$-error for the free surface elevation vs. $\Delta x$ for $k=1,2,3$ and $k=4$ at $t_{\textrm{max}}=0.02\,s$. }\label{order_h}
%\end{minipage}
\end{figure}

\begin{figure}
%\begin{minipage}{0.5\linewidth}
\centering
\begin{tikzpicture}
\begin{loglogaxis}[
    xlabel={$\Delta x$},	
    ylabel={$L^2$-error},
    xtick={0,-1,-2,-3,-4},
    legend pos=south east,
    ymajorgrids=true,
    grid style=dashed,
    yticklabel style = {font=\tiny},
    xticklabel style = {font=\tiny},
    %width=1\textwidth,
    %height=0.9\textwidth,
]
\addplot[
    color=blue,
    mark=square,
    ] table [y=a, x expr=1/\thisrow{dx}] {conv_u1.txt};
\addplot[
    color=red,
    mark=diamond,
    ] table [y=b, x expr=1/\thisrow{dx}] {conv_u1.txt};
\addplot[
    color=brown,
    mark=o,
    ] table [y=c, x expr=1/\thisrow{dx}] {conv_u1.txt};
\addplot[
    color=magenta,
    mark=star,
    ] table [y=d, x expr=1/\thisrow{dx}] {conv_u4.txt};
    
 \addplot[color=black,
	mark=no,]
	 table[ 
	 x expr=1/\thisrow{dx},
         y={create col/linear regression={y=a,
         variance list={10, 10, 10, 1000, 1000}}}]{conv_u1_line.txt}
         coordinate [pos=0.25] (A)
         coordinate [pos=0.1]  (B);
% save the slope parameter:
\xdef\slope{\pgfplotstableregressiona}
% draw the opposite and adjacent sides
% of the triangle
\draw (A) -| (B)
     node [pos=0.2,anchor=south]
     {\pgfmathprintnumber{\slope}};
      
\addplot[color=black,
	mark=no,]
	table[
         x expr=1/\thisrow{dx},
         y={create col/linear regression={y=b,
         variance list={10, 10, 10,1000, 1000}}}]{conv_u1_line.txt}
         coordinate [pos=0.1] (C)
         coordinate [pos=0.25]  (D)
   ;
      % save the slope parameter:
\xdef\slopeq{\pgfplotstableregressiona}
      % draw the opposite and adjacent sides
% of the triangle
\draw (C) -| (D)
     node [pos=0.2,anchor=north]
     {\pgfmathprintnumber{\slopeq}};
     %%%%%
     
\addplot[color=black, mark=no,]
	 table[
         x expr=1/\thisrow{dx},
         y={create col/linear regression={y=c,
         variance list={10, 10, 10, 1000, 1000}}}]{conv_u1_line.txt}
         coordinate [pos=0.22](E)
         coordinate [pos=0.1](F);
        % save the slope parameter:
        \xdef\slopec{\pgfplotstableregressiona}
        % draw the opposite and adjacent sides
        % of the triangle
        \draw (E) -| (F)
        node [pos=0.2,anchor=south]
       {\pgfmathprintnumber{\slopec}};
     
 \addplot[color=black, mark=no,]
	 table[
         x expr=1/\thisrow{dx},
         y={create col/linear regression={y=d,
         variance list={10, 10, 10, 1000}}}]{conv_u4_line.txt}
         coordinate [pos=0.1](G)
         coordinate [pos=0.23](H);
        % save the slope parameter:
        \xdef\sloped{\pgfplotstableregressiona}
        % draw the opposite and adjacent sides
        % of the triangle
        \draw (G) -| (H)
        node [pos=0.2,anchor=north]
       {\pgfmathprintnumber{\sloped}};

%\draw[-stealth] (A) -| (B);
    \legend{$k=1$,$k=2$,$k=3$,$k=4$}
\end{loglogaxis}

\end{tikzpicture}
\caption{Test 1 - $L^2$-error for the discharge vs. $\Delta x$ for $k=1,2,3$ and $k=4$ at $t_{\textrm{max}}=0.02\,s$}\label{order_q}
%\end{minipage}
\end{figure}
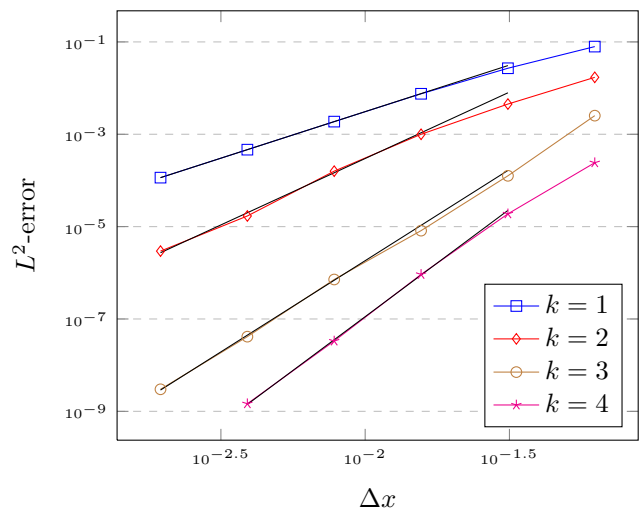

\begin{figure}[H]
\begin{center}
%\begin{fig}
\includegraphics[scale=0.25,angle=0]{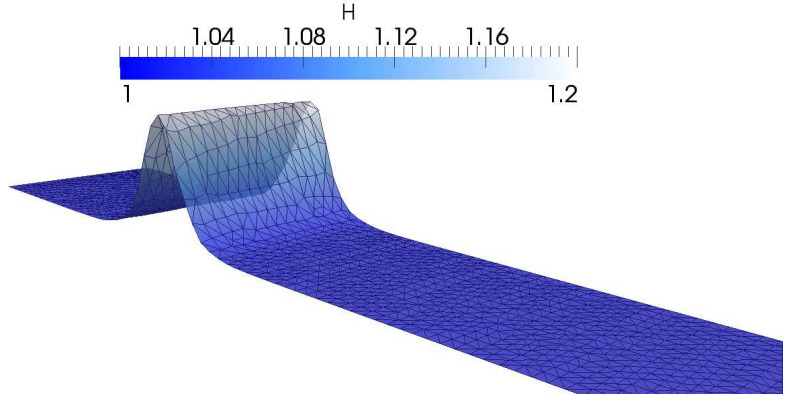} 
%\end{fig}
\caption{Test 2 - Solitary wave propagation over a flat bottom: initial free surface}
\label{soliton:init}
\end{center}
%\end{minipage}
\end{figure}
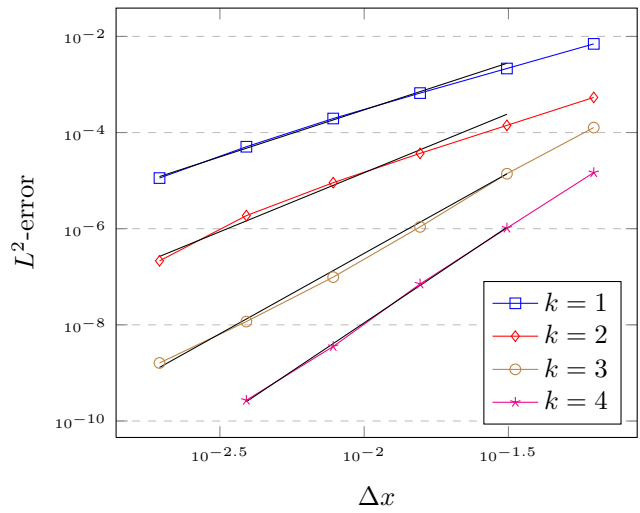
\begin{figure}[H]
%\begin{minipage}{0.5\linewidth}
\centering
\begin{tikzpicture}
\begin{loglogaxis}[
    xlabel={$\Delta x$},	
    ylabel={$L^2$-error},
    xtick={0,-1,-2,-3,-4},
    legend pos=south east,
    ymajorgrids=true,
    grid style=dashed,
    yticklabel style = {font=\tiny},
    xticklabel style = {font=\tiny},
    %width=1\textwidth,
    %height=0.9\textwidth,
]
\addplot[
    color=blue,
    mark=square,
    ] table [y=a, x expr=1/\thisrow{dx}] {t2_conv_h1.txt};
\addplot[
    color=red,
    mark=diamond,
    ] table [y=b, x expr=1/\thisrow{dx}] {t2_conv_h1.txt};
\addplot[
    color=brown,
    mark=o,
    ] table [y=c, x expr=1/\thisrow{dx}] {t2_conv_h1.txt};
\addplot[
    color=magenta,
    mark=star,
    ] table [y=d, x expr=1/\thisrow{dx}] {t2_conv_h4.txt};
    
 \addplot[color=black,
	mark=no,]
	 table[ 
	 x expr=1/\thisrow{dx},
         y={create col/linear regression={y=a,
         variance list={10, 10, 10, 10, 1000}}}]{t2_conv_h1_line.txt}
         coordinate [pos=0.25] (A)
         coordinate [pos=0.1]  (B);
% save the slope parameter:
\xdef\slope{\pgfplotstableregressiona}
% draw the opposite and adjacent sides
% of the triangle
\draw (A) -| (B)
     node [pos=0.1,anchor=south]
     {\pgfmathprintnumber{\slope}};
      
\addplot[color=black,
	mark=no,]
	table[
         x expr=1/\thisrow{dx},
         y={create col/linear regression={y=b,
         variance list={10, 10, 10, 10, 1000}}}]{t2_conv_h1_line.txt}
         coordinate [pos=0.1] (C)
         coordinate [pos=0.25]  (D)
   ;
      % save the slope parameter:
\xdef\slopeq{\pgfplotstableregressiona}
      % draw the opposite and adjacent sides
% of the triangle
\draw (C) -| (D)
     node [pos=0.2,anchor=north]
     {\pgfmathprintnumber{\slopeq}};
     %%%%%
     
\addplot[color=black, mark=no,]
	 table[
         x expr=1/\thisrow{dx},
         y={create col/linear regression={y=c,
         variance list={10, 1000, 10, 10}}}]{t2_conv_h1_line.txt}
         coordinate [pos=0.25](E)
         coordinate [pos=0.1](F);
        % save the slope parameter:
        \xdef\slopec{\pgfplotstableregressiona}
        % draw the opposite and adjacent sides
        % of the triangle
        \draw (E) -| (F)
        node [pos=0.2,anchor=south]
       {\pgfmathprintnumber{\slopec}};
     
 \addplot[color=black, mark=no,]
	 table[
         x expr=1/\thisrow{dx},
         y={create col/linear regression={y=d,
         variance list={10, 10, 10, 1000}}}]{t2_conv_h4_line.txt}
         coordinate [pos=0.1](G)
         coordinate [pos=0.23](H);
        % save the slope parameter:
        \xdef\sloped{\pgfplotstableregressiona}
        % draw the opposite and adjacent sides
        % of the triangle
        \draw (G) -| (H)
        node [pos=0.2,anchor=north]
       {\pgfmathprintnumber{\sloped}};

%\draw[-stealth] (A) -| (B);
    \legend{$k=1$,$k=2$,$k=3$,$k=4$}
\end{loglogaxis}

\end{tikzpicture}
\caption{Test 2 - $L^2$-error for the free surface elevation vs. $\Delta x$ for $k=1,2,3$ and $k=4$ at $t_{\textrm{max}}=0.2\,s$. }\label{order_h_solit}
%\end{minipage}
\end{figure}
\begin{figure}[H]
%\begin{minipage}{0.5\linewidth}
\centering
\begin{tikzpicture}
\begin{loglogaxis}[
    xlabel={$\Delta x$},	
    ylabel={$L^2$-error},
    xtick={0,-1,-2,-3,-4},
    legend pos=south east,
    ymajorgrids=true,
    grid style=dashed,
    yticklabel style = {font=\tiny},
    xticklabel style = {font=\tiny},
    %width=1\textwidth,
    %height=0.9\textwidth,
]
\addplot[
    color=blue,
    mark=square,
    ] table [y=a, x expr=1/\thisrow{dx}] {t2_conv_u1.txt};
\addplot[
    color=red,
    mark=diamond,
    ] table [y=b, x expr=1/\thisrow{dx}] {t2_conv_u1.txt};
\addplot[
    color=brown,
    mark=o,
    ] table [y=c, x expr=1/\thisrow{dx}] {t2_conv_u1.txt};
\addplot[
    color=magenta,
    mark=star,
    ] table [y=d, x expr=1/\thisrow{dx}] {t2_conv_u4.txt};
    
 \addplot[color=black,
	mark=no,]
	 table[ 
	 x expr=1/\thisrow{dx},
         y={create col/linear regression={y=a,
         variance list={10, 10, 10, 10, 1000}}}]{t2_conv_u1_line.txt}
         coordinate [pos=0.25] (A)
         coordinate [pos=0.1]  (B);
% save the slope parameter:
\xdef\slope{\pgfplotstableregressiona}
% draw the opposite and adjacent sides
% of the triangle
\draw (A) -| (B)
     node [pos=0.1,anchor=south]
     {\pgfmathprintnumber{\slope}};
      
\addplot[color=black,
	mark=no,]
	table[
         x expr=1/\thisrow{dx},
         y={create col/linear regression={y=b,
         variance list={10, 10, 10, 10, 10}}}]{t2_conv_u1_line.txt}
         coordinate [pos=0.25] (C)
         coordinate [pos=0.1]  (D)
   ;
      % save the slope parameter:
\xdef\slopeq{\pgfplotstableregressiona}
      % draw the opposite and adjacent sides
% of the triangle
\draw (C) -| (D)
     node [pos=0.2,anchor=south]
     {\pgfmathprintnumber{\slopeq}};
     %%%%%
     
\addplot[color=black, mark=no,]
	 table[
         x expr=1/\thisrow{dx},
         y={create col/linear regression={y=c,
         variance list={10, 1000, 10, 10}}}]{t2_conv_u1_line.txt}
         coordinate [pos=0.25](E)
         coordinate [pos=0.1](F);
        % save the slope parameter:
        \xdef\slopec{\pgfplotstableregressiona}
        % draw the opposite and adjacent sides
        % of the triangle
        \draw (E) -| (F)
        node [pos=0.2,anchor=south]
       {\pgfmathprintnumber{\slopec}};
     
 \addplot[color=black, mark=no,]
	 table[
         x expr=1/\thisrow{dx},
         y={create col/linear regression={y=d,
         variance list={10, 10, 10, 10}}}]{t2_conv_u4_line.txt}
         coordinate [pos=0.1](G)
         coordinate [pos=0.23](H);
        % save the slope parameter:
        \xdef\sloped{\pgfplotstableregressiona}
        % draw the opposite and adjacent sides
        % of the triangle
        \draw (G) -| (H)
        node [pos=0.2,anchor=north]
       {\pgfmathprintnumber{\sloped}};

%\draw[-stealth] (A) -| (B);
    \legend{$k=1$,$k=2$,$k=3$,$k=4$}
\end{loglogaxis}

\end{tikzpicture}
\caption{Test 2 - $L^2$-error for the discharge vs. $\Delta x$ for $k=1,2,3$ and $k=4$ at $t_{\textrm{max}}=0.2\,s$. }\label{order_q_solit}
%\end{minipage}
\end{figure}
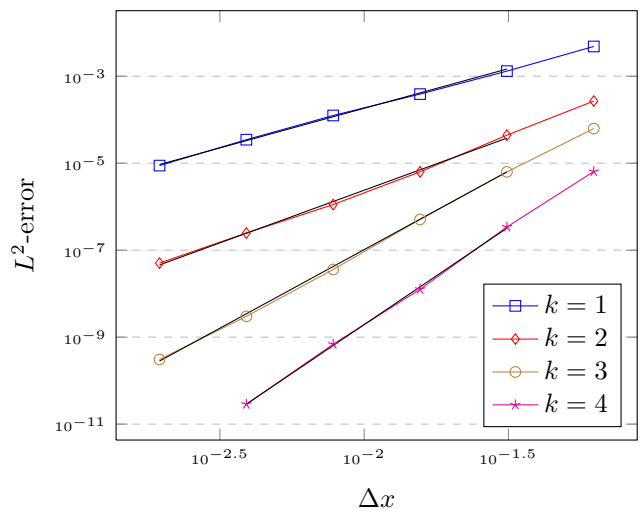

\subsection{Solitary wave propagation}
We consider now the time evolution of a solitary wave profile define as follows:
\begin{equation}\label{Soliton}
\left \lbrace \begin{array}{ll}
h(x,t) = h_{0} + \eps h_0 \,sech^{2}\left( \kappa (x - ct) \right) \, ,\\
u(x,t) = c \left( 1-\dfrac{h_{0}}{h(x,t)} \right) \, ,\end{array} \right.
\end{equation}
with $\kappa=\sqrt{\dfrac{3\eps}{4h_{0}^{2}(1+\eps)}}$ \, , \, and $c=\sqrt{gh_{0}(1+\eps)}$. Note that if such profiles are exact solutions of the original Green-Naghdi equations \eqref{eq6}, or equivalently \eqref{eq6imp} with $\alpha=1$, these are only solutions of the $\left(\mathcal{CG}_\alpha\right)$ model up to $\mathcal{O}(\mu^{2})$ terms. However, for small enough values of $\eps$, such profiles are expected to propagate over flat bottoms without noticeable deformations. We use a rectangular computational domain of $200\,m$ length and $25\,m$ width, and a relatively coarse unstructured mesh with faces' length ranging from $0.7\,m$ to $1.5\,m$. The reference water depth is set to $h_{0} = 1 m$, the relative amplitude is set to $\eps = 0.2$ and the initial free surface, shown on Fig. \ref{soliton:init}, is centred at $x_0 = 50\, m$.  Cross sections along the x-direction centerline, obtained with $k=3$ at several times during the propagation, are shown on Fig. \ref{Sol_fig2}.\\
To further investigate the convergence properties of our approach, we now consider a sequence of regular meshes with mesh size ranging from $2^{-4}\,m$ to $2^{-9}\,m$ and polynomial expansions of degrees ranging from $k=1$ to $k=4$. The corresponding $L^{2}$ errors computed at $t_{\textrm{max}}=0.2\,s$ are used to plot the convergence curves on Fig. \ref{order_h_solit} for the free surface elevation and Fig. \ref{order_q_solit} for the discharge. The corresponding convergence rates obtained by linear regression are also reported and we observe an irregular behavior with rates varying between $\mathcal{O}(\h^k)$ and $\mathcal{O}(\h^{k+1})$, generally close to $\mathcal{O}(\h^\frac{1}{2})$ but with a sub-optimal convergence observed on the free surface for $k=4$. Such a behavior is also observed with slightly larger amplitude waves.
\begin{figure}
\begin{center}
%\begin{fig}
\includegraphics[scale=0.4,angle=0]{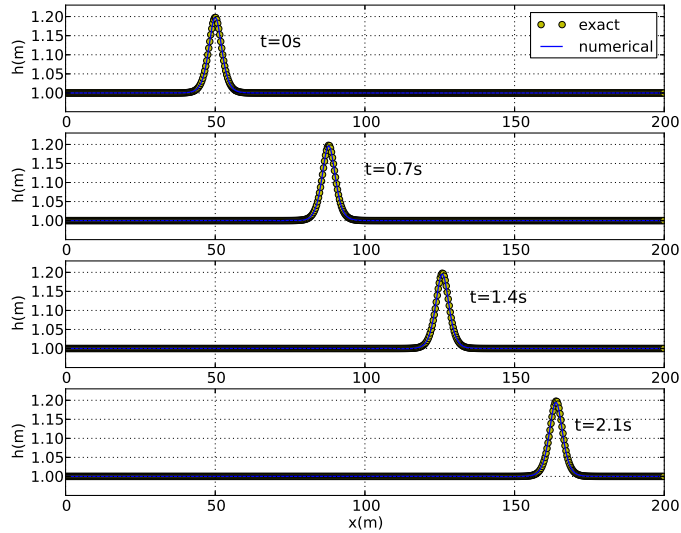} 
%\end{fig}
\caption{Test 2 - Solitary wave propagation: cross section of the free surface elevation at several times during the propagation.}
\label{Sol_fig2}
\end{center}
\end{figure}

\subsection{Run-up of a solitary wave}
We investigate now the ability of the scheme in handling dry areas with a test based on the experiments of \textit{Synolakis} \cite{Synolakis:1987p4276}. We study the propagation, shoaling, breaking and run-up of a solitary wave over a topography with constant slope $s=1/19.85$. The reference water depth is set to $h_0 = 1\,m$, and a solitary wave profile is considered as initial condition (\ref{Soliton}), with a relative amplitude $\eps=0.28$. The simulation involves a $50m \times 5m$ basin, regularly meshed with a space step $\Delta x = 0.25\,m$ and $k=2$. We show on Fig. \ref{Synolakis_1} some cross sections of the solution, taken at  various times during the propagation and compared with the experimental data.  
\begin{figure}
\begin{center}
\begin{fig}
\includegraphics[scale=0.65,angle=0]{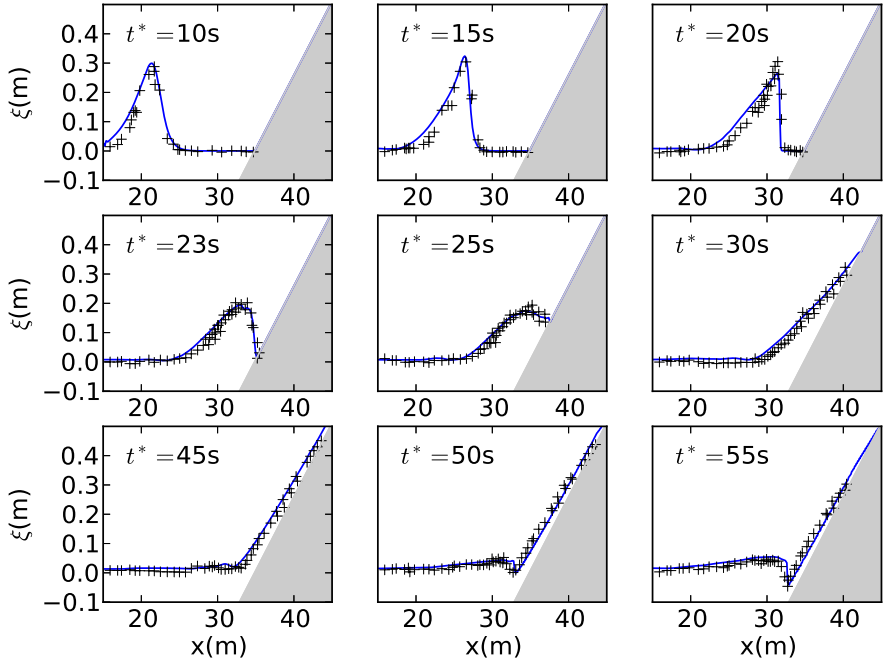} 
\end{fig}
\caption{Test 3 - Solitary wave breaking over a sloping beach: free surface profiles comparison
between numerical results (solid lines) and experimental data (crosses) at several times during the propagation ($t^{\ast} = t(g/h_0)^{1/2}$).}
\label{Synolakis_1}
\end{center}
\end{figure}
The wave breaking is identified approximately at $t^{\ast} = 17\,s$, with the normalized time $t^{\ast} = t(g/h_0)^{1/2}$, and occurs between gauges $\# 2$ and $\# 3$, which is in agreement with the experiment. The whole breaking process is well reproduced, as well as the subsequent run-up phenomena.

\subsection{Propagation of highly dispersive waves}
The dispersive properties of the numerical model are now assessed through the study of the propagation of periodic waves over a submerged bar, following the experiments of \textit{Dingemans} \cite{Dingemans:1994aa}. The computational domain is a $37.7m$ long and $0.8m$ wide basin. The topography is shown on Fig. \ref{Dingemans1}. The trapezoidal bar extends from $x=10\,m$ to $x=15\,m$ with slopes of $1/20$ at the front and $1/10$ at the back. The initial state is a flow at rest with a reference water depth of $h_0=0.4\,m$.
\begin{figure}
\begin{center}
\begin{fig}
\includegraphics[scale=0.3,angle=0]{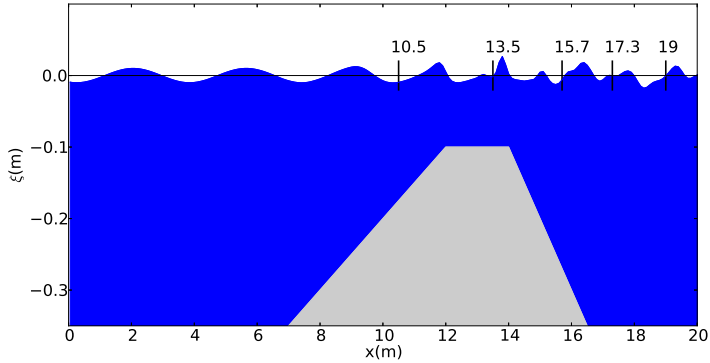}
\end{fig} 
\caption{Test 4 - Propagation of highly dispersive waves: sketch of the experiment configuration and location of the gauges used for test ${\#}A$.}
\label{Dingemans1}
\end{center}
\end{figure}
Periodic waves are generated at the left boundary, with an amplitude $a$, and a period $T$. Both generating and absorbing layers are set to $5\,m$ at the corresponding boundaries. The two following tests are carried out:
\begin{itemize}
\item[${\#}A$ :]  a=0.01\,m \quad , \quad T = 2.02\,s \quad , \quad \text{no wave - breaking}.
\item[${\#}B$ :]  a=0.025\,m \quad , \quad T = 2.51\,s \quad , \quad \text{wave - breaking}.
\end{itemize}
For both test ${\#}A$ and ${\#}B$, we set $k=2$ and use a regular triangulation obtained from rectangular elements $\Delta x = \Delta y = 0.125\,m$. The characteristics of the flow are quite complex here, notably due to the high non-linearities induced by the topography. The propagating waves first shoal and steepen over the submerged bar, generating higher-harmonics. These harmonics are progressively released on the downward slope, until encountering deeper waters again. In the first test, the initial amplitude is not large enough to trig the breaking of the waves and the Green-Naghdi equations are resolved in the whole domain. Fig. \ref{Dingemans1} indicates the location of the five wave gauges used for the first test to study the time-evolution of the free surface deformations. The results are plotted on Fig. \ref{Dingemans2}. We observe a very good agreement between analytical and experimental data for the first wave gages. As usual with this set-up, some discrepancies can be observed on the two last gages. Such behavior can be improved with the use of optimized models. In particular, the extension of the present DG approach to the \textit{3-parameters} model of \cite{lannes_marche:2014} is left for future works.
\begin{figure}
\begin{center}
%\begin{fig}
\includegraphics[scale=0.4,angle=0]{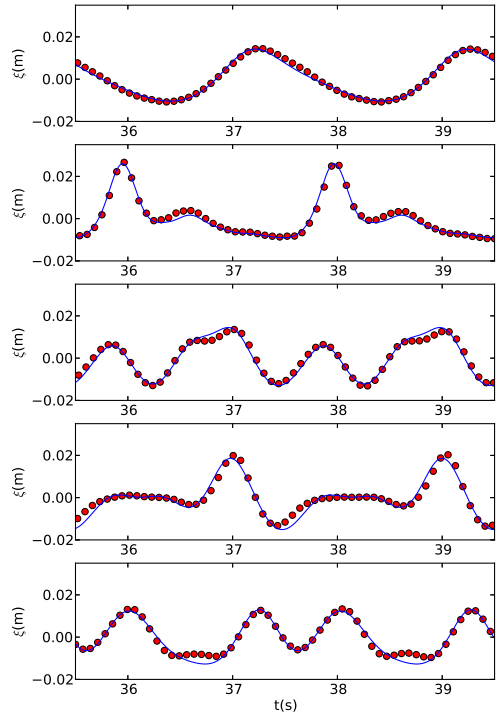}
%\end{fig} 
\caption{Test 4 - Propagation of highly dispersive waves: test ${\#}A$, free surface evolution at gauges. Numerical data are denoted in plain lines.}
\label{Dingemans2}
\end{center}
\end{figure}

The test ${\#}B$ corresponds to the experiments of \cite{BejiBattjes:1994aa}. The first reference gage is located at $x=6\,m$, followed by a series of seven gages regularly spaced between $x=11\,m$ and $x=17\,m$. Wave breaking is observed at the level of the flat part of the bump during the propagation and to stabilize the computation, the dispersive terms are turned-off in the vicinity of troubled cells, using the strategy described in \S\ref{Limitation}. Again, a comparison between the numerical results and the experimental data is performed and shown on Fig. \ref{Beji2} 

%\begin{figure}
%\begin{center}
%\begin{fig}
%\includegraphics[scale=0.5,angle=0]{Fig/CasTest4/Beji_break.eps}
%\end{fig} 
%\caption{Propagation of highly dispersive waves : test ${\#}B$. Illustration of the wave breaking strategy. The NSW equations are resolved in the red areas.}
%\label{Beji1}
%\end{center}
%\end{figure}

\begin{figure}
\begin{center}
\begin{fig}
\includegraphics[scale=0.35,angle=0]{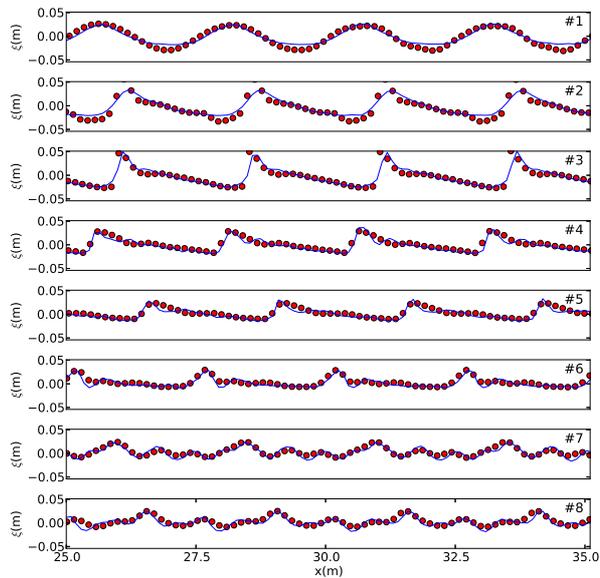}
\end{fig} 
\caption{Test 4 - Propagation of highly dispersive waves: test ${\#}B$, comparison with experimental data at gauges 1 to 8.}
\label{Beji2}
\end{center}
\end{figure}

\subsection{Periodic waves propagation over an elliptic shoal}
We study now the propagation of a train of monochromatic waves evolving over a varying topography, following the experiment of \textit{Berkhoff} \textit{et al} \cite{berkhoff:1982}. The experimental domain consists of a $20\,m$ wide and $22\,m$ long wave tank. The experimental model reproduces a seabed with a constant slope, forming an angle of $\alpha = 20^\circ$ with the $y$ axis, and deformed by a shoal with an elliptic shape, see Fig. \ref{Elliptic1}. The analytical profile for the topography in rotated coordinates $x_r = x cos(\alpha) - y sin(\alpha) , y_r =  x sin(\alpha) + y cos(\alpha)$ is given by $z = z_b + z_s$ where:
\begin{equation}
\begin{split}
z_b(x,y)&=\left\lbrace 
\begin{array}{ll} 
(5.82 + x_r)/50 \quad \text{if} \quad x_r \geq -5.82 \\
0 \quad \text{elsewhere}
\end{array} \right. \\
z_s(x,y)&=\left\lbrace 
\begin{array}{ll} 
-0.3 + 0.5 \sqrt{1 - \left( \dfrac{x_r}{3.75}\right)^{2} - \left( \dfrac{y_r}{5}\right)^{2} } \quad \text{if} \quad \left( \dfrac{x_r}{4}\right)^{2} + \left( \dfrac{y_r}{3}\right)^{2} \leq  1\\
0 \quad  \text{elsewhere}
\end{array} \right.
\end{split}
\end{equation}
The reference water depth is set to $h_0 = 0.45\,m$, and the corresponding computational domain has dimensions $[-10,12] \times [-10,10]$ (in $m$), with an extension of $5\,m$ at inlet and outlet boundaries for the generation of incident waves and their absorption. The periodic wave train has an amplitude of $a=0.0232\,m$ and a period of $T=1\,s$. Solid wall boundary conditions are used at $y=10\,m$ and $y=-10\,m$. We use an unstructured mesh of $25390$ elements which is refined in the region where the bottom variations are expected to have the greatest impact on the waves transformations (the smallest and largest mesh face's  lengths are respectively $0.2\,m$ and $0.5\,m$). The computations are performed with polynomial approximations of order $k=2$.
\begin{figure}
\begin{center}
%\begin{fig}
\includegraphics[scale=0.25,angle=0]{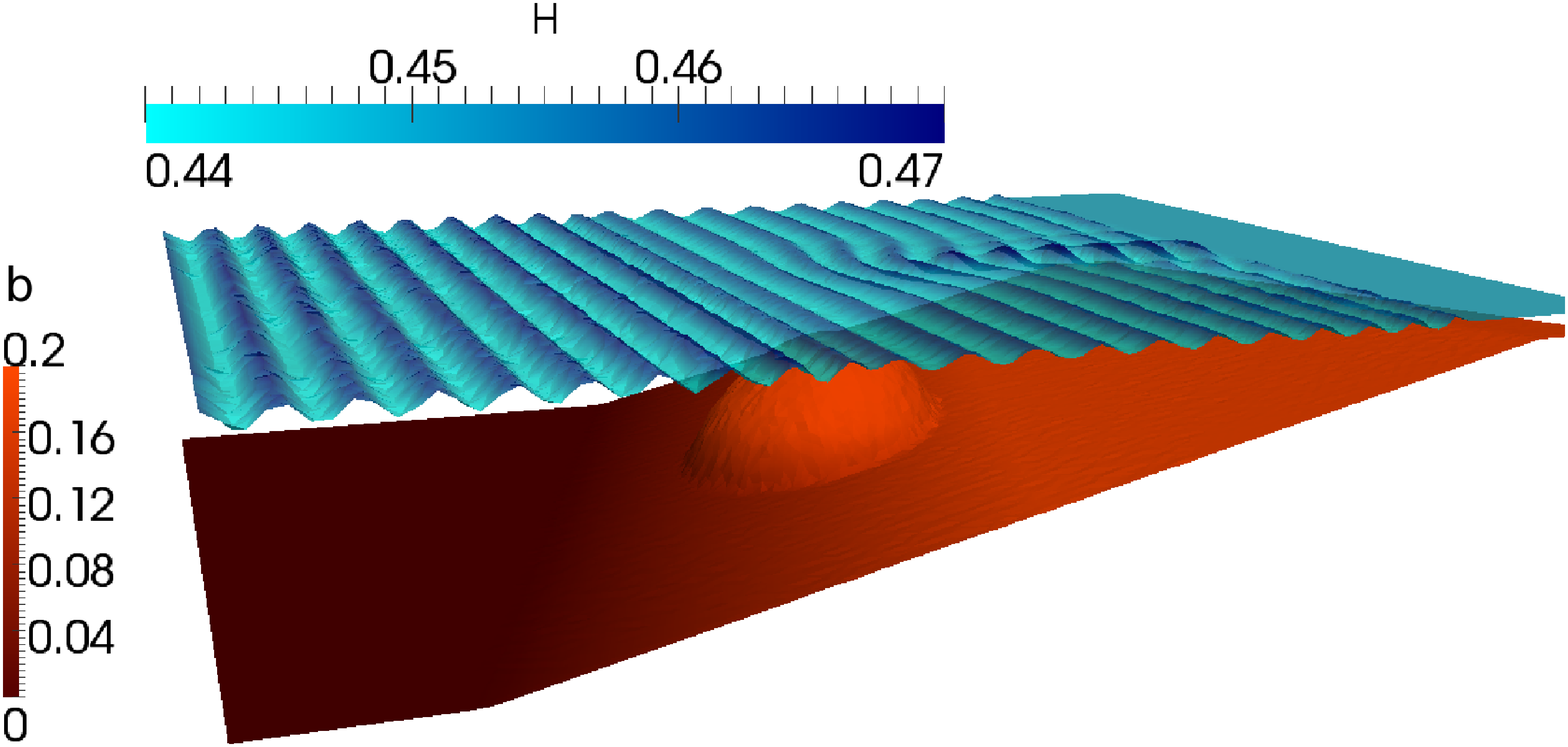}
%\end{fig} 
\caption{Test 5 - Periodic waves propagation over an elliptic shoal: topography and view of the free surface.}
\label{Elliptic1}
\end{center}
\end{figure}
%\begin{figure}
%\begin{center}
%%\begin{fig}
%\includegraphics[scale=0.25,angle=0]{mesh.eps}
%%\end{fig} 
%\caption{Wave train propagation over an elliptic shoal : mesh of the computational domain.}
%\label{Elliptic_mesh}
%\end{center}
%\end{figure}
The data issued from the experiment provides the normalized time-averaged of the free surface measured along several cross sections. In our numerical experiment we focus on the following ones:
\begin{equation}
\begin{split}
&\text{Section 2 :} \quad x = 3\,m \quad , \quad -5\,m \leq y \leq 5\,m \, ,\\
&\text{Section 3 :} \quad x = 5\,m \quad , \quad -5\,m \leq y \leq 5\,m \, ,\\
&\text{Section 5 :} \quad x = 9\,m \quad , \quad -5\,m \leq y \leq 5\,m \, ,\\
&\text{Section 7 :} \quad y = 0\,m \quad , \quad \quad 0\,m \leq x \leq 10\,m \, ,
\end{split}
\end{equation}
allowing a good coverage of the computational domain. Time series of the free surface elevation are hence recorded along these sections, from $t=30\,s$ to $t=50\,s$ and post-processed by mean of the \textit{zero up-crossing} method to isolate single waves and compute the mean wave elevation. The results, classically normalized by the incoming wave amplitude $a$, are shown in Fig. \ref{Elliptic2} and compared with the data.
\begin{figure}
\begin{center}
\begin{fig}
\includegraphics[scale=0.4,angle=0]{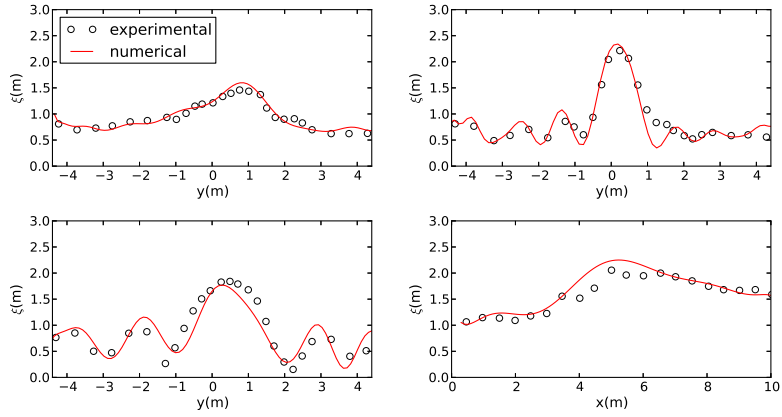}
\end{fig} 
\caption{Test 5 - Periodic waves propagation over an elliptic shoal: comparison with experimental data along the four sections.}
\label{Elliptic2}
\end{center}
\end{figure}

\begin{figure}
\begin{center}
\begin{fig}
\includegraphics[scale=0.2,angle=0]{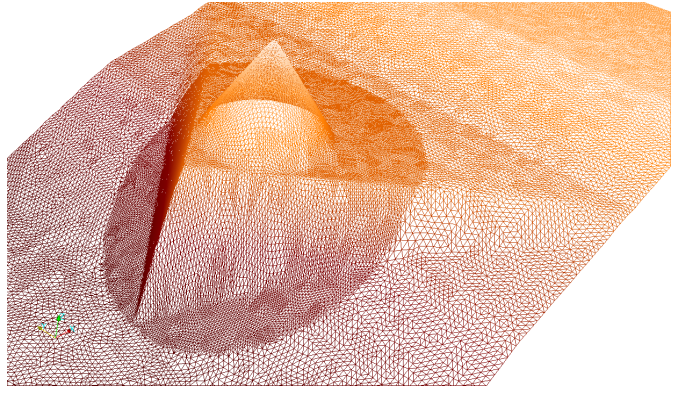} \, 
\end{fig} 
\caption{Test 6 - Solitary wave propagation over a 3d reef: overview of the mesh in the refined area.}
\label{topo2}
\end{center}
\end{figure}
\subsection{Solitary wave propagation over a 3d reef}
The following test case is based on the experiments of \cite{Swigler:2009aa} and allows to study some complex wave's interactions such as shoaling, refraction, reflection, diffraction, breaking and moving shoreline in a fully two-dimensional context. The experimental domain is a $48.8\;m$ large and $26.5\,m$ wide basin. The topography is a triangular shaped shelf with an island feature located at the offshore point of the shelf. The island is a cone of $6\,m$ diameter and $0.45\,m$ height is also placed on the apex, between $x=14\,m$ and $x=20\,m$ (see Fig. \ref{topo1}). During the experiments, free surface information was recorded via $9$ wave gauges, and the velocity information was recorded with $3$ Acoustic Doppler Velocimeters (ADVs). The complete set up can be found in \cite{Swigler:2009aa} but we also specify the gauge locations in the legend of Fig. \ref{t5gauge_eta}.\\
\begin{figure}[H]
\begin{center}
\begin{fig}
\includegraphics[scale=0.2,angle=0]{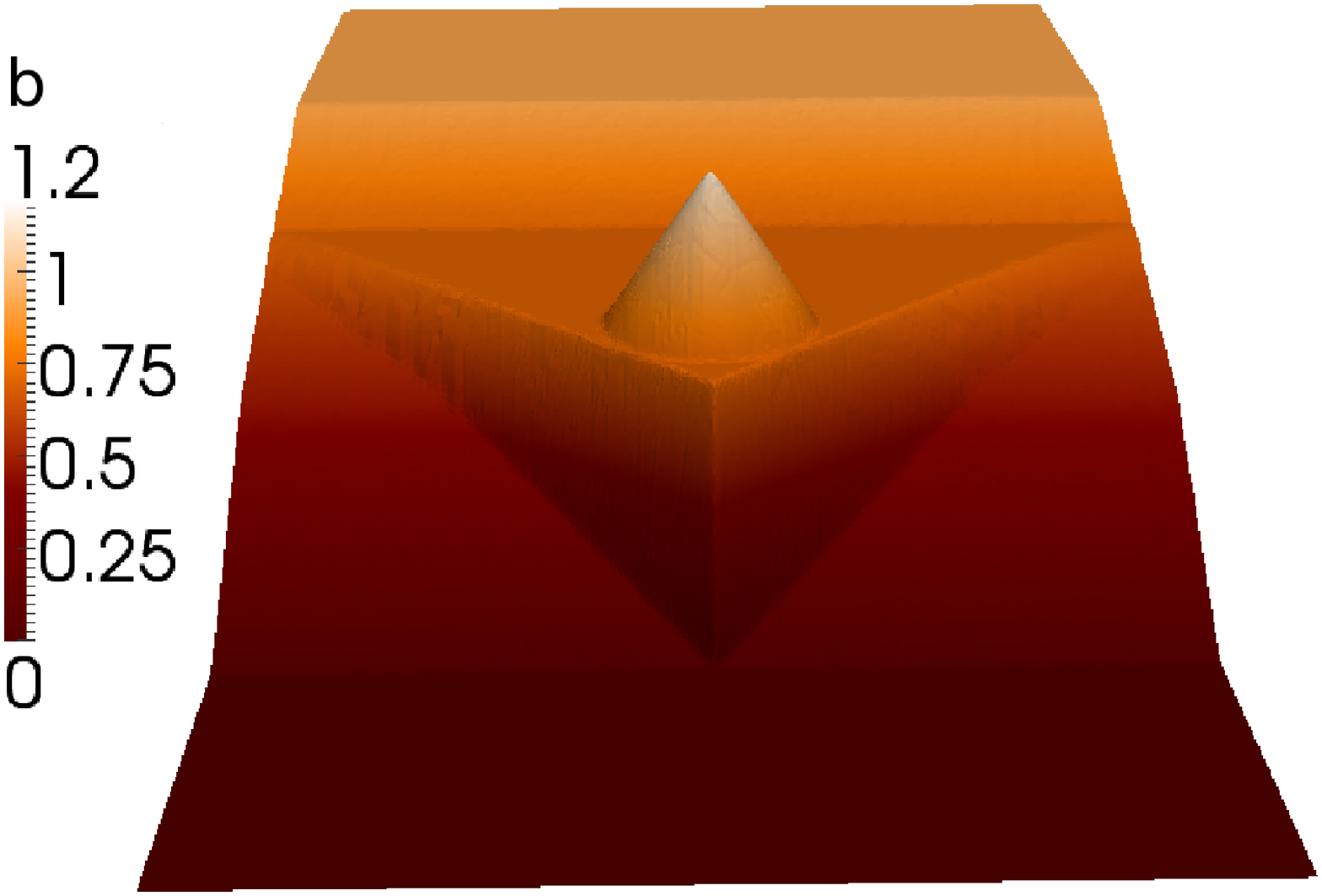} \, \includegraphics[scale=0.2,angle=0]{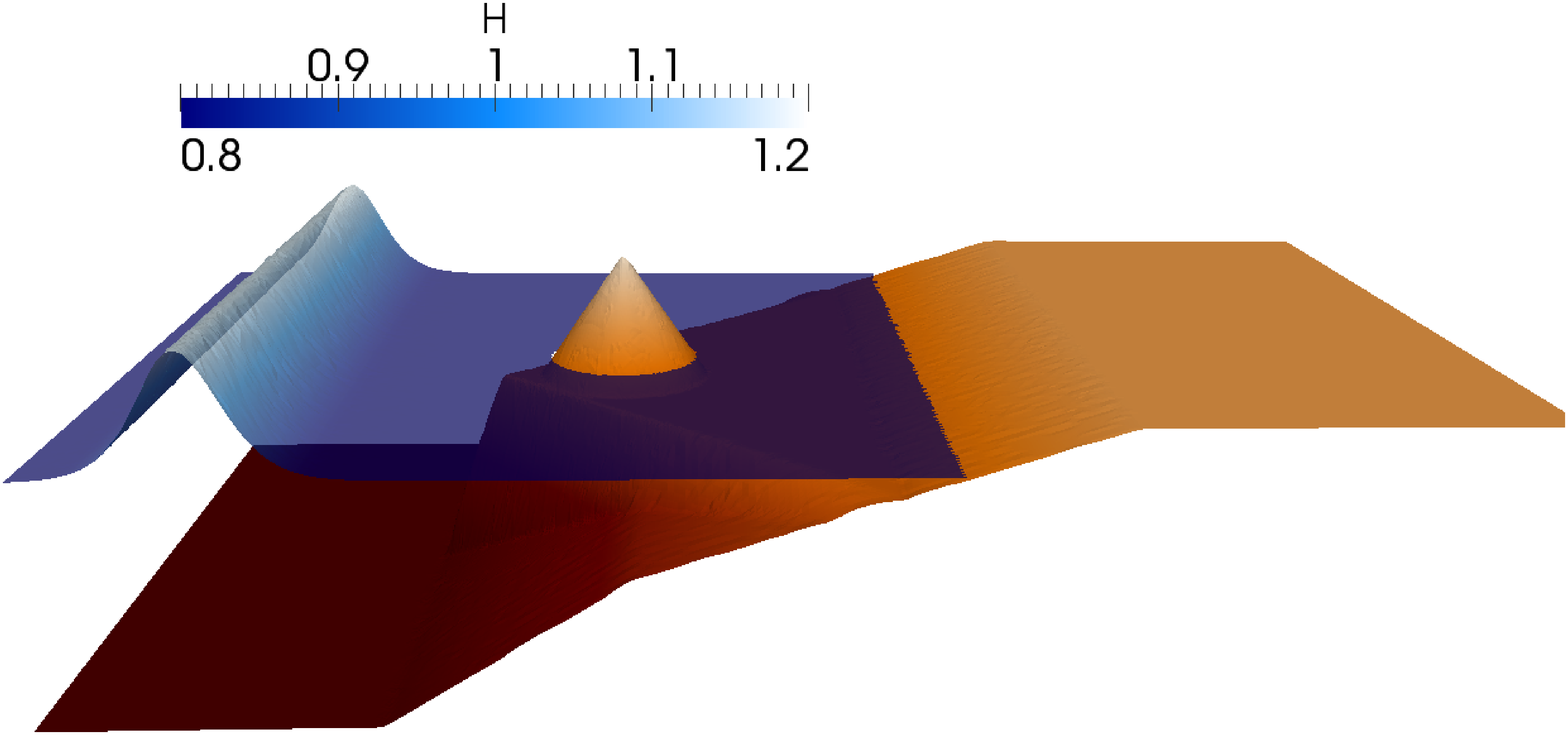}
\end{fig} 
\caption{Test 6 - Solitary wave propagation over a 3d reef: 3d view of the topography and initial condition.}
\label{topo1}
\end{center}
\end{figure}
Such a geometry motivates the use of an unstructured mesh of $45000$ elements refined in the vicinity of the cone, leading to a smallest face's length of $0.12\,m$ and a largest of $0.51\,m$, see Fig. \ref{topo2}. We set the polynomial order to $k=2$.  The reference water depth is $h=0.78\,m$ and we study the propagation and transformations of a solitary wave of relative amplitude $\eps = 0.5$. In addition to the complex processes already stated above, the transformations are quite nonlinear, making the following test particularly interesting. We show on Fig.\ref{t5AC3d1}-\ref{t5AC3d2} some snapshots of the free surface evolution before and after crossing the shelf's apex. We observe that the wave breaks at the apex slightly before $t = 5\,s$, wrapping the cone around $t=6.5\,s$. We also observe refracted waves from the reef slopes and diffracted waves around the cone which converge at the rear side at approximatively $t=8.5\,s$. The snapshots at $t=11.5\,s$ shows the run-up on the beach. Note that according with the positivity preservation property shown in the previous section, the simulation remains stable even with the use of a third order scheme.
We also provide a comparison between the numerical results and the data taken from the experiments for both free surface evolutions at the wave gages locations on Fig. \ref{t5gauge_eta} and the velocity at the ADVs locations, see Fig.\ref{t5gauge_u}. We highlight that even with this relatively moderate level of refinement, the characteristics of the flow are well resolved, especially when compared with those provided for instance in recent works using more refined meshes, see for instance\cite{Kazolea:2014kx, Roeber20121, shi:2012}.

\begin{figure}[H]
\begin{center}
\begin{fig}
\includegraphics[scale=0.18,angle=0]{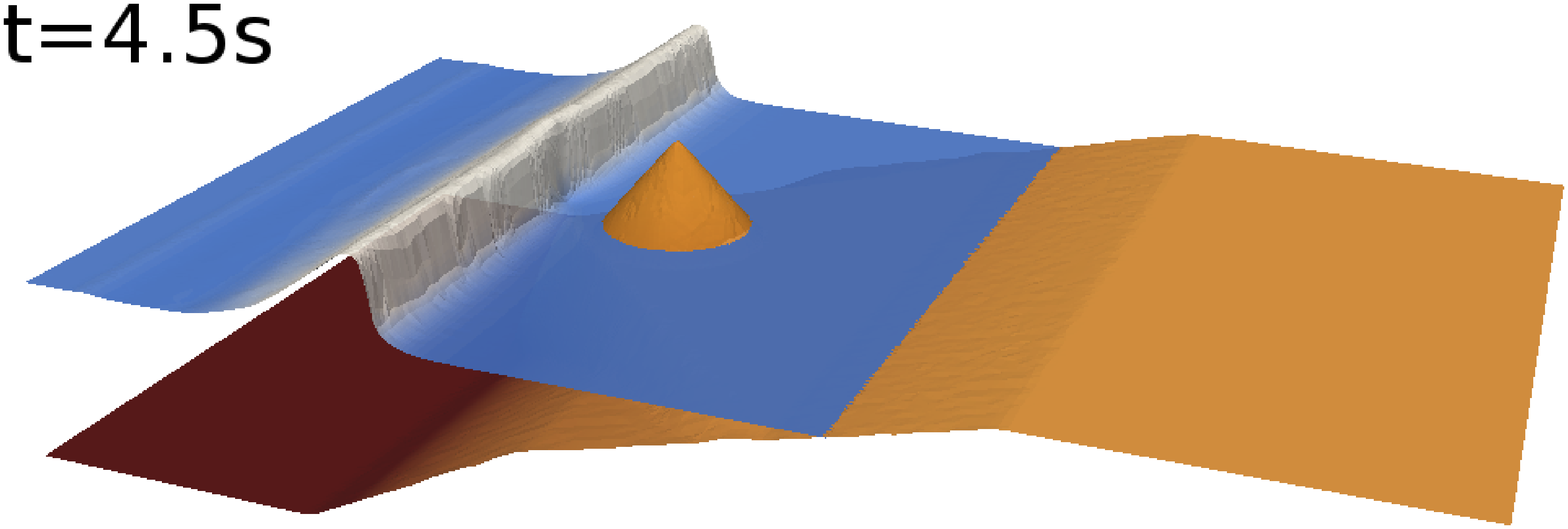}  \includegraphics[scale=0.18,angle=0]{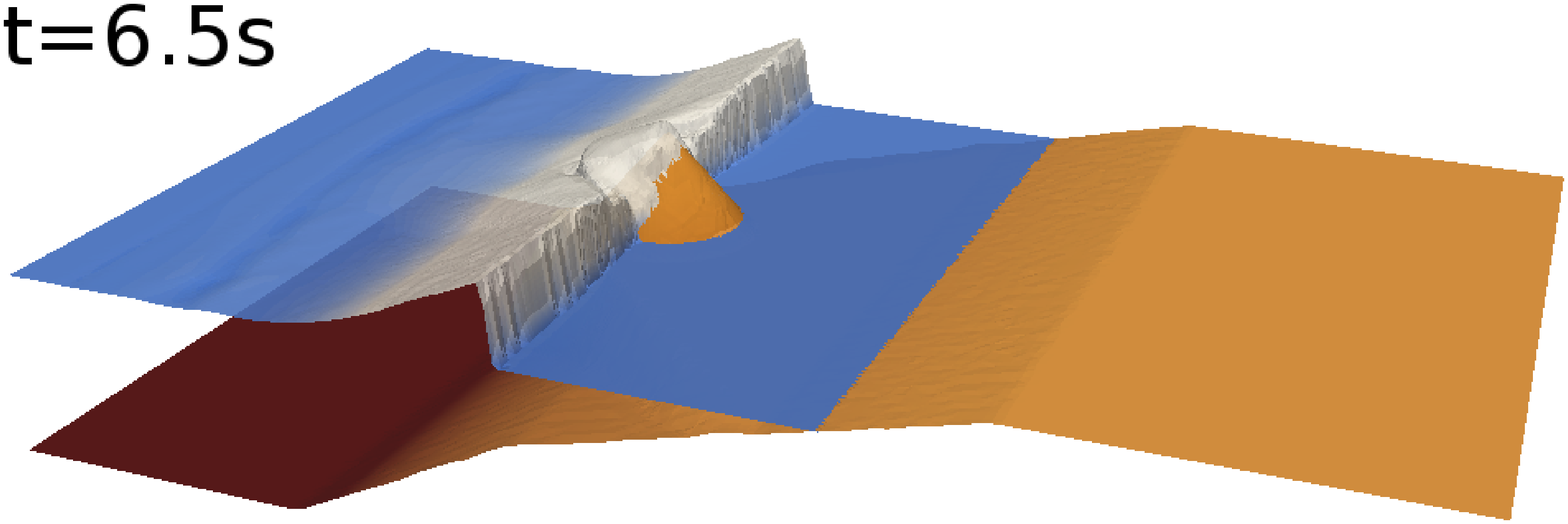}\\
\includegraphics[scale=0.18,angle=0]{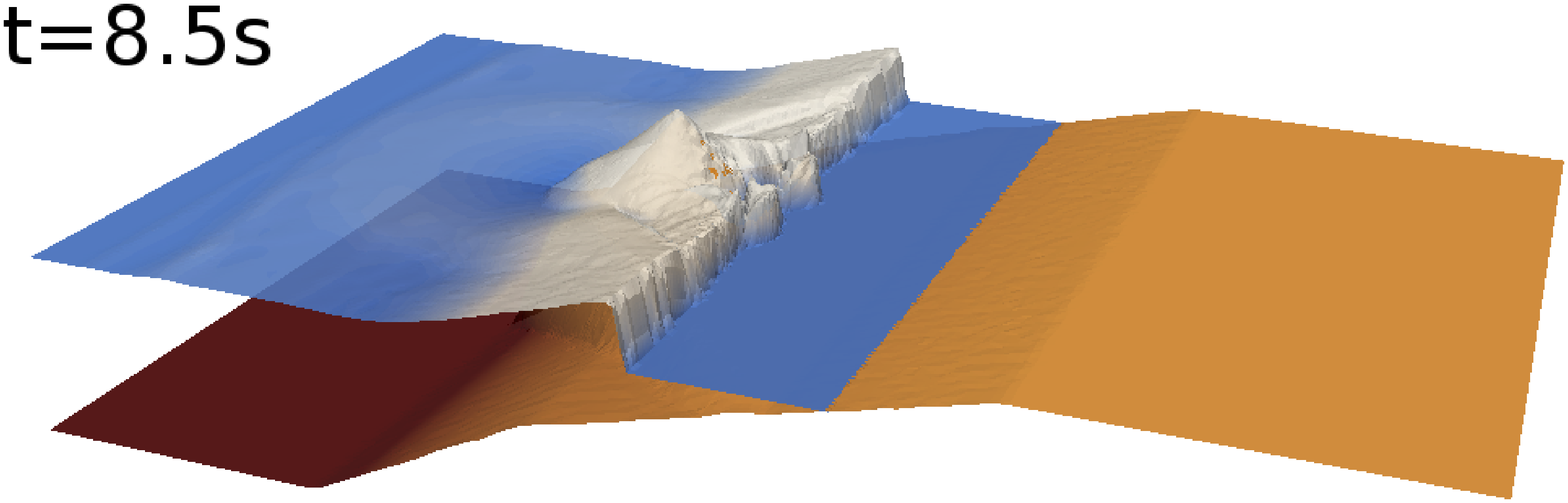}  \includegraphics[scale=0.18,angle=0]{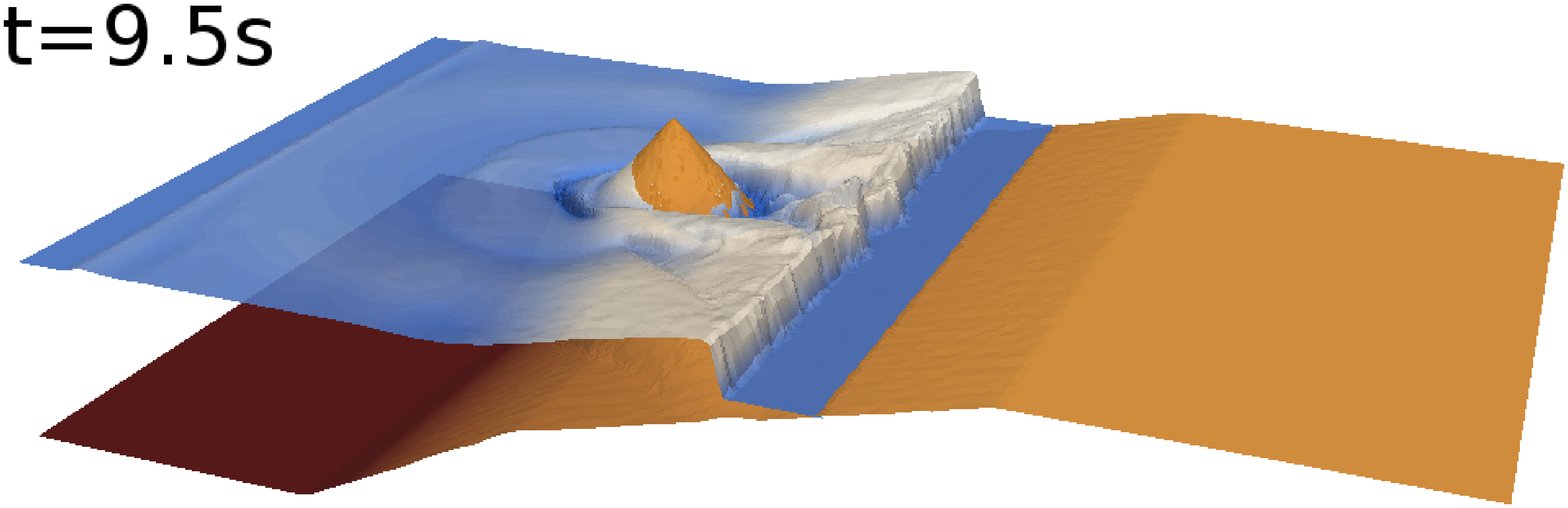}\\
\end{fig} 
\caption{Test 6 - Solitary wave propagation over a 3d reef - free surface at times $t=4.5, 6.5, 8.5$ and $9.5\,s$.}
\label{t5AC3d1}
\end{center}
\end{figure}
\begin{figure}[H]
\begin{center}
\begin{fig}
\includegraphics[scale=0.18,angle=0]{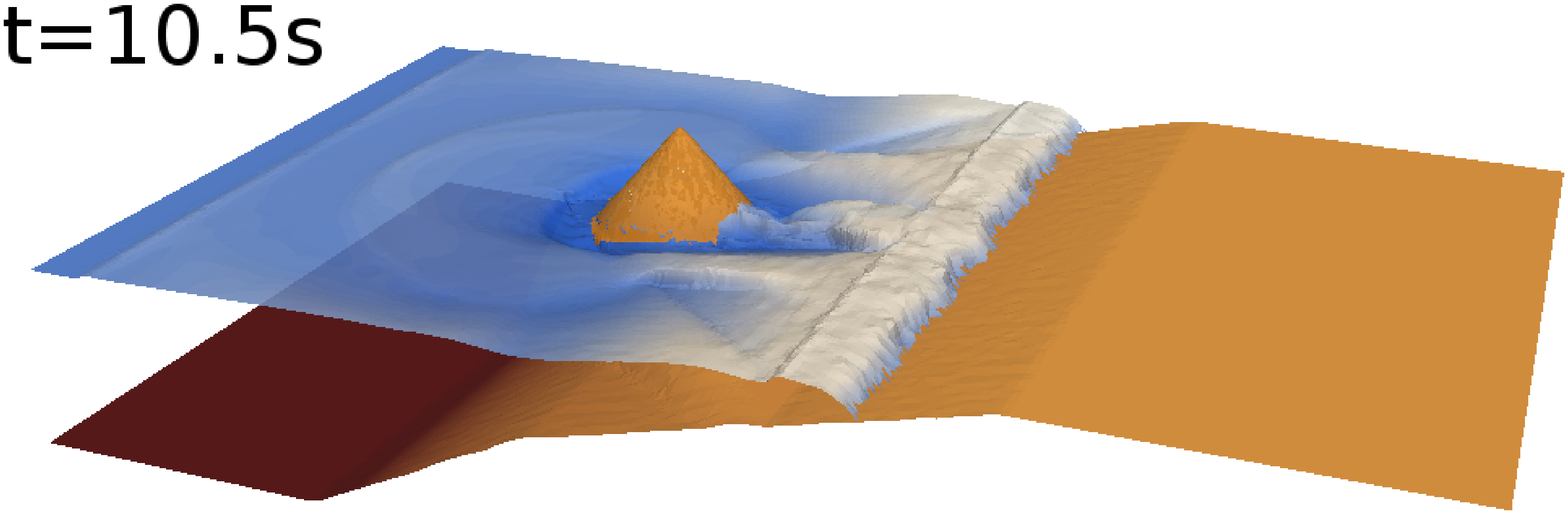}  \includegraphics[scale=0.18,angle=0]{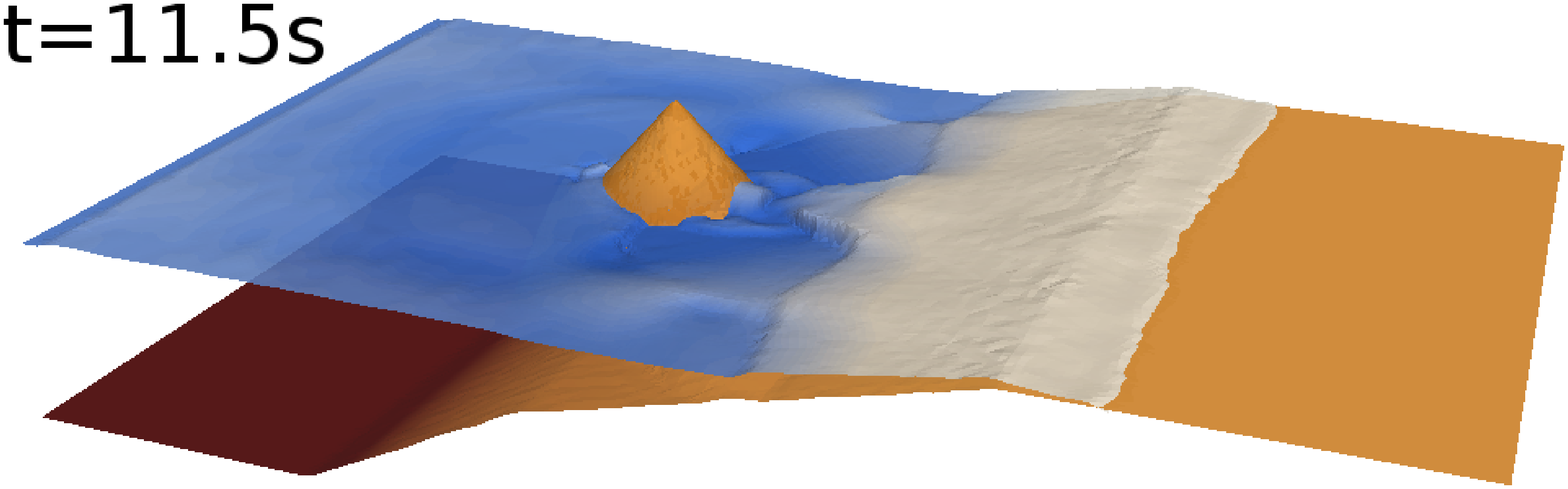}\\
\includegraphics[scale=0.18,angle=0]{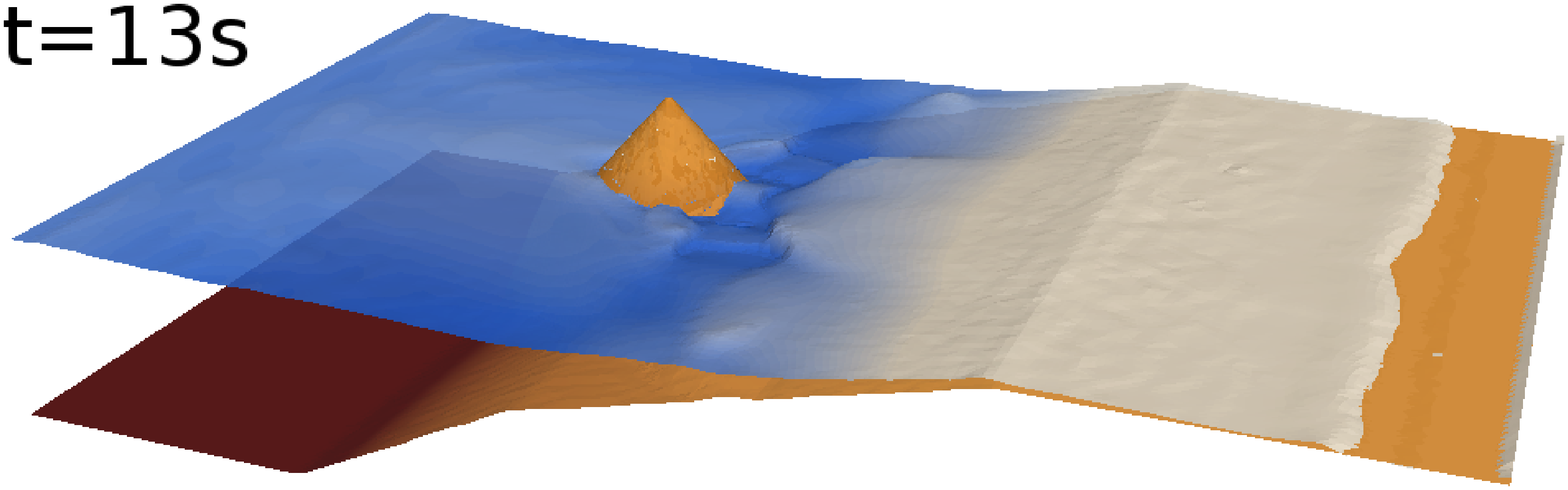}  \includegraphics[scale=0.18,angle=0]{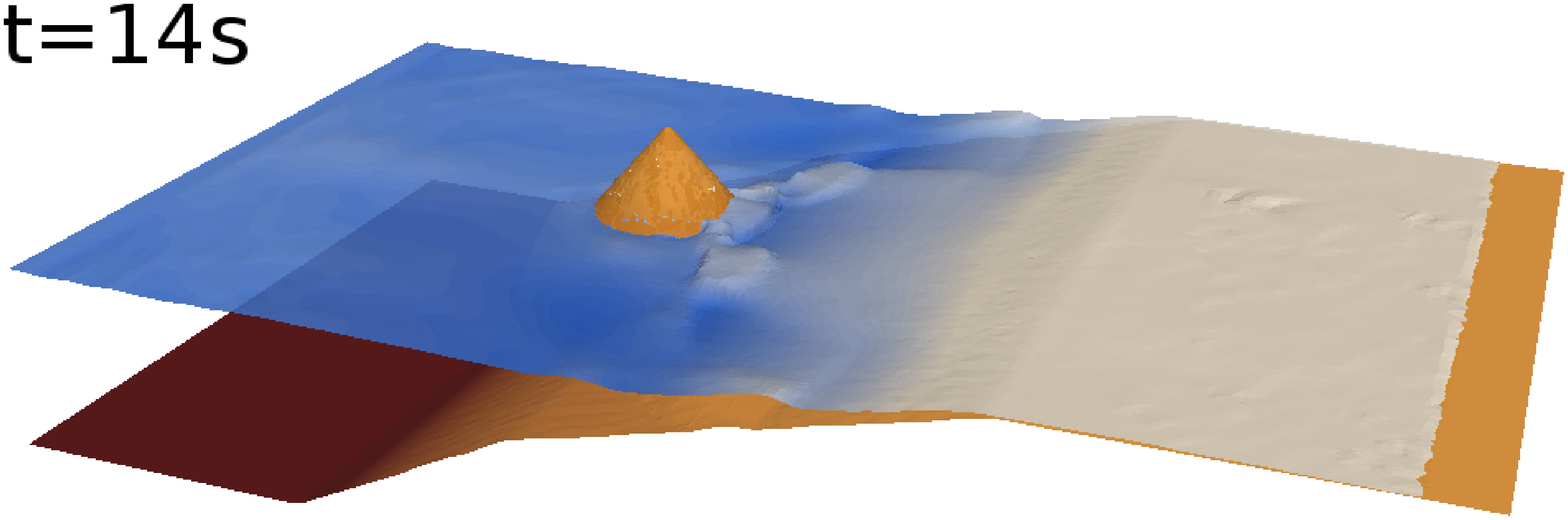}
\end{fig} 
\caption{Test 6 - Solitary wave propagation over a 3d reef - free surface at times $t=10.5, 11.5, 13$ and $14\,s$.}
\label{t5AC3d2}
\end{center}
\end{figure}

%\begin{figure}[H]
%\begin{center}
%\begin{fig}
%\includegraphics[scale=0.16,angle=0]{t5_B.eps}  \includegraphics[scale=0.16,angle=0]{t5_D.eps}\\
%\includegraphics[scale=0.16,angle=0]{t5_F.eps}  \includegraphics[scale=0.16,angle=0]{t5_H.eps}
%\end{fig} 
%\caption{Test 6 - Solitary wave propagation over a 3d reef: free surface at t=4.5,9,10 and 13s.}
%\label{t53d}
%\end{center}
%\end{figure}
\begin{figure}
\begin{center}
\begin{fig}
\includegraphics[scale=0.35,angle=0]{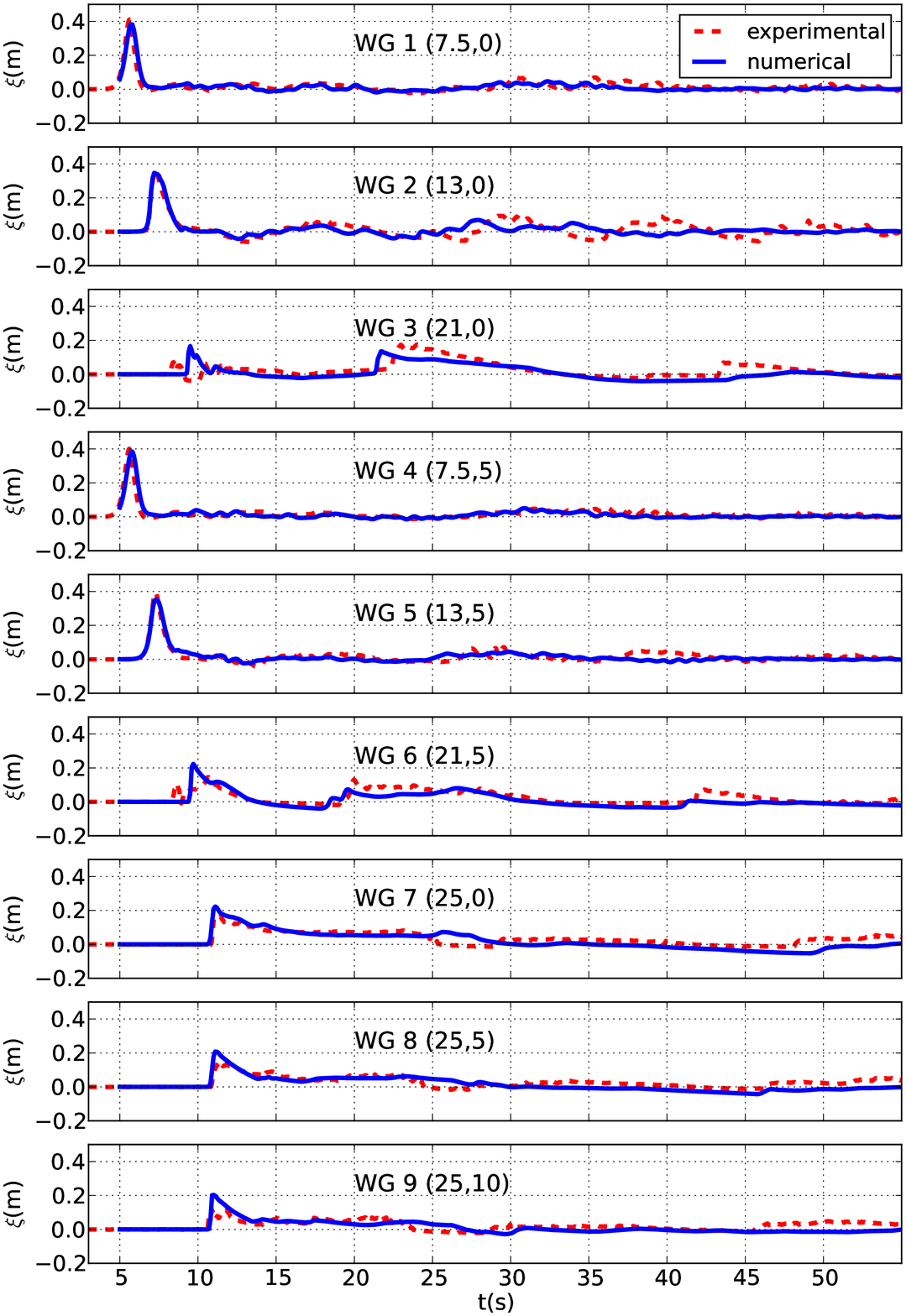}
\end{fig} 
\caption{Test 6 - Solitary wave propagation over a 3d reef - Comparison with experimental solution at gauges 1 to 9.}
\label{t5gauge_eta}
\end{center}
\end{figure}
\begin{figure}
\begin{center}
\begin{fig}
\includegraphics[scale=0.35,angle=0]{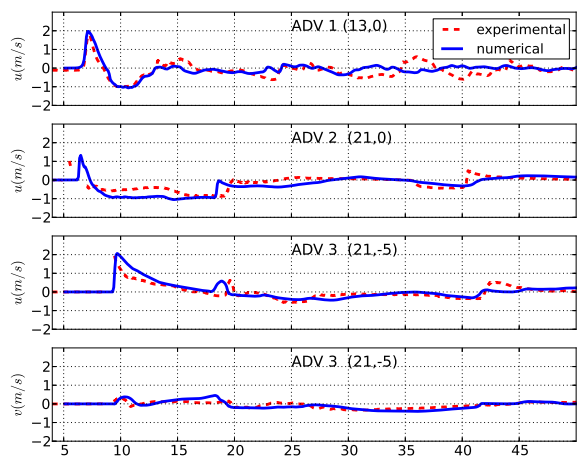}
\end{fig} 
\caption{Test 6 - Solitary wave propagation over a 3d reef: $x$ and $y$ velocity components - Comparison with experimental solution at several ADVs.}
\label{t5gauge_u}
\end{center}
\end{figure}

\section{Conclusion}
In this work, we introduce for the first time a fully discontinuous Galerkin formulation designed to approximate the solutions of a multi-dimensional Green-Naghdi model on arbitrary unstructured simplicial meshes. The underlying elliptic-hyperbolic decoupling approach, already investigated in \cite{DuranMarche:2014ab} for the $d=1$ case, allows to regard the dispersive correction simply as an algebraic source term in the NSW equations. Such dispersive correction is obtained as the solution of \textit{scalar} elliptic second order sub-problems, which are approximated using a mixed formulation and \textit{LDG} fluxes. Additionally, the preservation of motionless steady states and of the positivity of the water height, under a suitable time step restriction, are ensured  for the whole formulation and for any order of approximation.\\
These properties are assessed through several benchmarks, involving nonlinear wave transformations and severe occurrence of dry areas, and some additional convergence studies are performed. Although the obtained numerical results clearly show very promising abilities for the study of nearshore flows, there are still several issues that need to be addressed in future works. In particular, we use a low-brow approach to stabilize the computations in the vicinity of broken waves. If such a method provides results which are in good agreement with experimental data for the cases under study in this work, we are still working to extend the method of \cite{tissier2} to the $d=2$ case in a robust way.\\
Another important issue may be related to the computational cost of our approach. If we believe in the potential benefits of the use of such a non-conforming high approach, especially concerning the robust approximations of solutions involving strong gradients occurring in the vicinity of breaking waves and possibly strong topography variations, it clearly leads to larger algebraic systems than those obtained with classical continuous approximations. Future works will therefore be devoted to go further in the development and optimization of our approach, and in particular to the reduction of the associated computational cost.

\paragraph{Acknowledgements}
The second author acknowledges partial support from the ANR-13-BS01-0009-01 BOND.

%\newpage
{%

  %\bibliographystyle{plain}
  %\bibliography{mybib}
}
\end{document}